\documentclass{article}

     \usepackage[nonatbib,preprint]{neurips_2019}

\usepackage[utf8]{inputenc} 
\usepackage[T1]{fontenc}    
\usepackage{hyperref}       
\usepackage{url}            
\usepackage{booktabs}       
\usepackage{amsfonts}       
\usepackage{amsmath,amsthm,amssymb}
\usepackage{mathrsfs}
\usepackage{enumitem}
\usepackage{algorithm}
\usepackage{algorithmic}
\usepackage{mcr}
\usepackage{graphicx}
\usepackage{subcaption}
\usepackage{epstopdf}
\title{Wasserstein Distributionally Robust Inverse Multiobjective Optimization}

\author{
	Chaosheng Dong\footnote{} \\
	Amazon\\
	\texttt{chaosd@amazon.com} \\
	\And
	Bo Zeng \\
	University of Pittsburgh\\
	\texttt{bzeng@pitt.edu} \\
}

\begin{document}
	
	\maketitle
	
	\begin{abstract}
	Inverse multiobjective optimization provides a general framework for the unsupervised learning task of inferring parameters of a multiobjective decision making problem (DMP), based on a set of observed decisions from the human experts. However, the performance of this framework relies critically on the selection of appropriate decision making structure, a set of observed decisions that are sufficient and of high qualities, and a parameter space that contains enough information about the DMP. To hedge against the uncertainties in the hypothetical DMP, the data, and the  parameter space, we investigate in this paper the distributionally robust approach for inverse multiobjective optimization. Specifically, we leverage the Wasserstein metric to construct a ball centering at the empirical distribution of these decisions. We then formulate a Wasserstein distributionally robust inverse multiobjective optimization problem (WRO-IMOP) that minimizes a worst-case expected loss function, where the worst case is taken over all distributions in the Wasserstein ball. We  show that the excess risk of the WRO-IMOP estimator has a sub-linear convergence rate. Furthermore, we propose a semi-infinite reformulations of the WRO-IMOP and develop a cutting-plane algorithm that converges to any $\delta$-optimal solution in finite iterations. Finally, we demonstrate the effectiveness of our method on both a synthetic multiobjective quadratic program and a real world portfolio optimization problem.
\end{abstract}

\section{Introduction}
\label{intro: wimop}

Assuming decision makers are rational when making decisions, inverse optimization has been proposed to  infer or to learn their utility functions and restrictions based on observed decisions or behaviors. 
Traditional inverse optimization considers the decision making problem (DMP) with a single objective function \cite{ahuja2001inverse,keshavarz2011imputing,bertsimas2015data,chan2019inverse,barmann2017emulating,dong2018ioponline}.
Such consideration is suitable for a single decision maker or when decision makers are of the same type. Nevertheless, when multiple observations are available, especially when they are collected from different decision makers, it might not be a reasonable framework to estimate DMP's parameters. As shown in \cite{dong2020imop}, the derived inference could be drastically different from decision makers' actual
intention. 

In the real world, it is basically ubiquitous that decisions are trade-offs among multiple criteria.   Hence, inverse optimization considering multiobjective optimization functions, referred to as inverse multiobjective optimization problem (IMOP), provides a more appropriate and compelling tool to learn humans' decision making scheme or emulate their behaviors \cite{dong2020imop,dong2018inferring}. Specifically, the learner observes a set of \emph{Pareto optimal} decisions, i.e., $ \{\mfy_{i}\}_{i \in [N]} $ from the following  multiobjective optimization problem (MOP):
\begin{align*}
\begin{array}{llll}
\min\limits_{\mfx \in \bR^{n}} & \{f_{1}(\mfx,\theta),f_{2}(\mfx,\theta),\ldots,f_{p}(\mfx,\theta)\} \\
\;s.t. & \mfx \in X(\theta),
\end{array}
\end{align*}
where $ \theta $ is the true but unknown parameter for the expert's multiobjective DMP. Indeed, inverse multiobjective optimization answers the following fundamental question -- \textit{how do we learn $ \theta $ across multiple objective functions given $ \{\mfy_{i}\}_{i \in [N]} $?}

This tool is generally applicable to learn the underlying multiobjective decision making schemes in many scenarios. They, once obtained, would presumably play critical roles in various aspects, such as assisting organizations in designing products and providing services to customers. For example, a portfolio manager typically applies the Markovitz mean-variance model—risk and profit are two objectives to optimize—to make investment decisions \cite{markowitz1952portfolio}. To automate the portfolio investment, one needs to learn the key parameter of this model, e.g., the risk aversion score or the expected returns of the assets, by observing the portfolio manager's historical investment records.

We note that the effectiveness and accuracy of IMOP relies critically on the selection of appropriate decision making structure, a set of observed decisions that are sufficient and of high qualities, and a parameter space that contains as much information about the objective functions or constraints as possible. In practice, however, it is highly unlikely that all of these critical factors would be satisfied. For example,  outliers in a limited amount of decisions would render the empirical distribution of decisions deviate from the true distribution, and thus significantly weaken the predictive power of the IMOP estimator. We mention that this issue is not unique to IMOP model and one can observe similar findings in other inverse optimization models  \cite{keshavarz2011imputing,bertsimas2015data,chan2014generalized,esfahani2017data,birge2017inverse,barmann2017emulating,chan2018trade,dong2018inferring,dong2018ioponline,chan2019inverse,chan2020inverse}.

To hedge against these uncertainties contained in the hypothetical decision making model, the data, and the selected parameter space, we investigate  the distributionally robust approach for inverse multiobjective optimization. More specifically, motivated by \cite{Shafieezadeh2015wdrlogistic,gao2016distributionally,esfahani2018distributionally}, etc., we use the Wasserstein metric \cite{villani2008optimal} to construct the uncertainty set centered at the empirical distribution of the observed decisions. Subsequently, we propose a distributionally robust inverse multiobjective optimization program that minimizes the worst-case risk of loss, where the worst case is taken over all distributions in the uncertainty set. By such a distributionally robust framework, we aim to bridge the discrepancy between the lack of certainties in the information and the expectation for the accurate prediction of human's behavior.

\subsection{Related Work}
Our work  is most related  to \cite{dong2018inferring,dong2020imop} that propose a general framework of using inverse multiobjective optimization to infer the objective functions or constraints of the multiobjective DMP, based on observations of Parete optimal solutions. Research in \cite{dong2018inferring,dong2020imop} presents the framework of empirical risk minimization for this unsupervised learning task, which generally works well when there are few uncertainties in the model, data or hypothetical parameter space. In contrast, we believe that those uncertainties root inherently in the applications of inverse multiobjective optimization, and we aim to hedge against their influences by adopting the distributionally robust optimization paradigm based on Wasserstein metric.  We demonstrate both theoritically and experimentally that our method has out-of-sample performance guarantees under uncertainties.

Our work draws inspirations from \cite{esfahani2017data} that develops a distributionally robust approach for inverse optimization to infer the utility function from sequentially arrived observations.
They aim to mitigate the poor performance of inverse optimization models \cite{ahuja2001inverse,keshavarz2011imputing,bertsimas2015data,chan2019inverse,barmann2017emulating,dong2018ioponline} when the learner has imperfect information.
They show that the associated distributionally robust inverse optimization approach offers out-of-sample performance guarantees under such a situation.
However, their approach is specially designed for the simpler case where the DMP has only one objective function. Differently, our approach considers a more complex situation and is suitable when the decision making problem has multiple objectives. Moreover, instead of using the suboptimality loss function, we consider another one that would better capture the learner's purpose to predict the decision maker's decisions. Due to the nonconvex nature of our loss function, extensive efforts are made to develop the algorithm for solving the resulting nonconvex minmax program.

\subsection{Contributions}
We summarize our contributions as follows:
\begin{itemize}
	\item We present a novel Wasserstein distributionally robust framework for constructing inverse multiobjective optimization estimator.  We use the prominent Wasserstein metric to construct the uncertainty set centering at the empirical distribution of observed decisions.
	\item
	We show that the proposed framework has statistical performance guarantees, and the excess risk of the distributionally robust inverse multiobjective optimization estimator would converge to zero with a sub-linear rate as the number of observed decisions approaches to infinity.
	\item We reformulate the resulting minmax problem as a semi-infinite program and develop a cutting-plane algorithm which converges to any $\delta$-solution in finite iterations. We demonstrate the effectiveness of our method on both a multiobjective quadratic program and a portfolio optimization problem.
\end{itemize}

\section{Problem setting}
\label{sec:problem setting for wimop}
Throughout this paper we will use $ \one $ to denote the vector of all ones, $ \zero $ to denote the vector of all zeros, and $ \bI $ to denote the identity matrix. For any $ n \geq 1 $, the set of integers $ \{1,\ldots,n\} $ is denoted by $ [n] $. The Euclidean norm is denoted by $ || \cdot||_{2} $. The inner product of two vectors $ \mfx $ and $ \mfy $ is defined by $ <\mfx, \mfy> = \mfx^{T}\mfy$. Before going further into different sections, we provide the main notations that will be used throughout the paper. The detailed descriptions of the notations could be found in Table \ref{table: list of notation}.

\begin{table}
	\caption{List of notation}
	\label{table: list of notation}
	\begin{tabular}{ll}
		\toprule
		{Indices: } \\
		$n$ & Dimension of the decision variable in \ref{mop} \\
		$n_\theta$ & Dimension of $\theta$ in \ref{mop} \\
		$p$ & The number of objective functions in \ref{mop}\\
		$q$ & The number of constraints in \ref{mop}\\
		{Sets: } \\
		$ \mathcal{P} $ & The Wasserstein ambiguity set \\
		$\Theta$ & The parameter space for $ \theta $\\
		$S(w,\theta)$ & The set of optimal solutions for \ref{weighting problem}\\
		$\mathscr{W}_{p}$ & The standard $ (p-1) $-simplex defined as $ \{ w\in \bR^{p}_{+} : \; \one^{T}w = 1 \} $\\
		$\mathscr{W}_{p}^{+}$ & The interior of the standard $ (p-1) $-simplex\\
		$X(\theta)$ & The feasible region for \ref{mop}\\
		$X_{P}(\theta)$  & The Pareto optimal set for \ref{mop} \\
		$ \mathcal{Y} $  & The support  of the observations \\
		{Model parameters: } \\
		$\theta$ & The model parameter that determines \ref{mop}\\
		$  \mathbf{x} \in \mathbb{R}^{n}  $    & Decision variable in \ref{mop} \\
		$ Q \in \mathbb{S}^{n \times n}_{+} $ & Covariance matrix of the assets in Portfolio optimization  \\
		$ \mathbf{c} \in \mathbb{R}^{n} $    & Expected profit vector \\
		$ A \in \mathbb{R}^{q \times n} $    & Linear constraint matrix \\
		$ \mathbf{b} \in \mathbb{R}^{q}$ & Right-hand side in constraints \\
		$\mathbf{y} \in \mathbb{R}^{n}$ & observed noisy decision (investment portfolio) \\
		$\mathbf{u} \in \mathbb{R}^{q}$ & dual variables for \ref{mop}\\
		$\mathbf{z} \in \{0, 1\}^{q} $  & binary variables \\
		$K$ & The number of weight samples in \eqref{surrogate loss function} \\
		$N$ & The number of observations \\
		$l(\mfy,\theta)$ & The \ref{loss function} \\
		$l_{K}(\mfy,\theta)$ & The \ref{surrogate loss function} \\
		$B$ & The upper bound for the $ L_2 $ norm  of the elements in the feasible region $ X(\theta) $ \\
		$R$ & The upper bound for the $ L_2 $ norm  of the elements in the support $ \mathcal{Y} $ of the observations \\
		$D$ & The upper bound for the $ L_2 $ norm  of the elements in $ \Theta $\\
		$\epsilon$ & The radius of the 1-Wasserstein ball\\
		$\kappa$ & The perturbation constant regarding to the change of objective values\\
		$\lambda_{l}$ & $ f_{l} $ in \ref{mop} is strongly convex with parameter $ \lambda_{l} $ \\
		$\lambda$ &  The minimum among  $\{\lambda_{l}\}_{l \in[p]}$ \\
		\bottomrule
	\end{tabular}
\end{table}

\subsection{Multiobjective decision making problem}
We consider a family of parametrized multiobjective decision making problems of the form
\begin{align*}
\label{mop}
\tag*{MOP}
\begin{array}{llll}
\min\limits_{\mfx \in \bR^{n}} & \big\{f_{1}(\mfx,\theta),f_{2}(\mfx,\theta),\ldots,f_{p}(\mfx,\theta)\big\} \\
\;s.t. & \mfx \in X(\theta)
\end{array}
\end{align*}
where $ p \geq 2 $ and $f_{l} : \bR^{n}\times\bR^{n_{\theta}} \mapsto \bR $ for each $ l \in [p] $. Assume paratmeter $\theta \in \Theta \subseteq \mathbb{R}^{n_\theta}$. We denote the vector of objective functions by $ \mathbf{f}(\mfx,\theta) = (f_{1}(\mfx,\theta),f_{2}(\mfx,\theta),\ldots,f_{p}(\mfx,\theta))^{T} $. Assume $ X(\theta) = \{\mfx \in \bR^{n}: \mathbf{g}(\mfx,\theta) \leq \zero, \mfx \in \bR^{n}_{+}\} $, where $ \mathbf{g}(\mfx,\theta) = (g_{1}(\mfx,\theta), \ldots, g_{q}(\mfx,\theta))^{T} $ is another vector-valued function with $g_{k} : \bR^{n}\times\bR^{n_{\theta}} \mapsto \bR  $ for each $ k \in [q] $.

\begin{definition}[Pareto optimality]
	For fixed $ \theta $, a decision vector $\mfx^{*} \in X(\theta)$ is said to be Pareto optimal or efficient if there exists no other decision vector $\mfx \in X$ such that $f_{i}(\mfx,\theta) \leq f_{i}(\mfx^{*},\theta)$ for all $i \in [p]$, and $f_{k}(\mfx,\theta) < f_{k}(\mfx^{*},\theta)$ for at least one  $k \in [p]$.
\end{definition}
In the study of multiobjective optimization, the set of all Pareto optimal solutions is denoted by $ X_{P}(\theta) $ and called the Pareto optimal set.
One common way to derive a Pareto optimal solution is to solve a problem with a single objective function constructed by the weighted sum of original objective functions \cite{gass1955computational}, i.e.,
\begin{align}
\label{weighting problem}
\tag*{WP}
\begin{array}{llll}
\min & w^{T}\mathbf{f}(\mfx,\theta)\\
\;s.t. & \mfx \in X(\theta),
\end{array}
\end{align}
where $ w = (w^{1},\ldots,w^{p})^{T}$ is the nonnegative weight vector in the $ (p-1) $-simplex $\mathscr{W}_{p}  \equiv \{ w\in \bR^{p}_{+} : \; \one^{T}w = 1 \}$.
We denote $S(w,\theta)$ the set of optimal solutions of \ref{weighting problem},  i.e.,
\[S(w,\theta) = \arg\min_{\mfx}\left\{ w^{T}\mathbf{f}(\mfx,\theta): \mfx \in X(\theta) \right\}.\]

Let $ \mathscr{W}_{p}^{+} = \{ w\in \bR^{p}_{++} : \; \one^{T}w = 1 \} $. Following from Theorem 3.1.2 of \cite{miettinen2012nonlinear}, we have:
\begin{proposition}\label{prop:positive}
	If $ \mfx \in S(w,\theta) $ and $ w \in \mathscr{W}_{p}^{+} $, then $ \mfx \in X_{P}(\theta) $.
\end{proposition}

The next result from Theorem 3.1.4 of \cite{miettinen2012nonlinear} states that all the Pareto optimal solutions can be found by the weighting method for convex \ref{mop}.
\begin{proposition}\label{weight_convex}
	Assume that \ref{mop} is convex. If $\mfx\in X$ is an Pareto optimal solution, then there exists a weight vector $w \in \mathscr{W}_{p}$ such that $ \mfx $ is an optimal solution of \ref{weighting problem}.
\end{proposition}

By Propositions \ref{prop:positive} - \ref{weight_convex}, we can summarize the relationship between $ S(w,\theta)$ and $ X_{P}(\theta) $ as follows.
\begin{corollary}\label{coro:inclusion}
	For convex \ref{mop},
	\begin{align*}
	\bigcup_{w\in \mathscr{W}_{p}^{+} }S(w,\theta) \subseteq X_{P}(\theta) \subseteq \bigcup_{w\in \mathscr{W}_{p}}S(w,\theta)
	\end{align*}
\end{corollary}

In the following, we make a few assumptions to simplify our understanding, which are actually mild and appear frequently in the literature.

\begin{assumption}
	\label{assumption:convex_setting}
	Set $ \Theta$ is a convex compact set in $ \bR^{n_{\theta}} $.  There exists $D>0$ such that $\sup_{\theta\in\Theta}\norm{\theta} \leq D$. In addition, $\mathbf{f}(\mfx,\theta)$ and $\mathbf{g}(\mfx,\theta)$ are convex in $\mfx$ for each $ \theta \in \Theta $.
\end{assumption}

\subsection{Inverse multiobjective optimization}
Consider a learner who has access to decision makers' decisions, but does not have the full knowledge of the underlying decision making model. In the inverse multiobjective optimization model, the learner aims to learn the parameter $ \theta $ in \ref{mop} from observed noisy decisions only, and no information regarding decision makers' preferences over multiple objective functions is available. We denote $ \mfy $ the observed noisy decision that might carry measurement error or is generated with a bounded rationality of the decision maker, i.e., being suboptimal. Throughout the paper we assume that $\mfy$ is a random variable distributed according to an unknown distribution $\bP_{\mfy}$ supported on $\cY$.

Ideally, the learner would aim to learn $\theta$ by deriving parameter values that minimize the distance between the noisy decision and the predicted decision derived with those values. Without knowing decision makers' preferences over multiple objective functions, however, the learner cannot predict a desired decision even when $\theta$ is given. To address such a challenge, we begin with a discussion on the construction of an appropriate loss function for the inverse multiobjective optimization problem \cite{dong2018inferring,dong2020imop}.

\begin{definition}
	Given a noisy decision $\mfy$ and a hypothesis $\theta$, the loss function is defined as the minimum (squared) distance between $\mfy$ and the Pareto optimal set $X_{P}(\theta)$:
	\begin{align}
	\label{loss function}
	\tag*{loss lunction}
	l(\mfy,\theta) = \min_{\mfx \in X_{P}(\theta)} \norm{\mfy - \mfx}^{2}.
	\end{align}
\end{definition}

For a general \ref{mop}, however, there might exist no explicit way to characterize the Pareto optimal set $ X_{P}(\theta) $. Hence, an approximation approach to practically describe this set can be adopted. Following from Corollary \ref{coro:inclusion}, a sampling approach can be adopted to generate $w_{k}\in \mathscr{W}_{p}$ for each $k\in[K]$ and approximate $X_{P}(\theta)$ as $ \bigcup_{k\in [K]}  S(w_{k},\theta) $.
\begin{definition}
	Given a noisy decision $\mfy$ and a hypothesis $\theta$, the \textit{surrogate loss function} is defined as
	\begin{align}
	\label{surrogate loss function}
	\tag*{surrogate loss}
	l_{K}(\mfy,\theta) = \min_{\mfx \in \bigcup\limits_{k\in [K]}  S(w_{k},\theta)} \norm{\mfy - \mfx}^{2}.
	\end{align}
	By using binary variables,  this surrogate loss function can be converted into the \textit{surrogate loss problem}.
	\begin{align}
	\label{surrogate loss problem}
	\begin{array}{llll}
	l_{K}(\mfy,\theta) &= & \min\limits_{z_j\in \{0,1\}} \ \norm{\mfy - \sum\limits_{k\in[K]}z_{k}\mfx_{k}}^{2} \vspace{1mm}\\
	& \mbox{s.t.} & \mfx_{k} \in S(w_{k},\theta),  & \forall k \in [K], \vspace{1mm} \\
	&  & \sum\limits_{k\in[K]}z_{k} = 1.
	\end{array}
	\end{align}
\end{definition}

As $ \mfx_{k} \in S(w_{k},\theta) $  indicates that we sample data points from the Pareto optimal set $X_{P}(\theta)$, we refer to the constraint as the Pareto optimal constraint for each $ k \in [K] $. Constraint $\sum_{k\in[K]}z_{k} = 1$ ensures that exactly one of Pareto optimal solutions will be chosen to measure the distance from $\mfy$ to $ X_{P}(\theta) $. Hence, solving this optimization problem identifies some $w_k$ with $k\in [K]$ such that the corresponding Pareto optimal solution $S(w_{k},\theta)$ is closest to $\mfy$.

\begin{remark}
	It is guaranteed that no Pareto optimal solution will be excluded if all weight vectors in $\mathscr{W}_{p}$ are enumerated. As it is practically infeasible due to computational intractability, we can control the number of sampled weights $ K $ to achieve a desired tradeoff between the approximation accuracy and computational efficacy. Certainly, if the computational power is strong, we would suggest to draw a large number of weights evenly in $\mathscr{W}_{p}$ to avoid any bias. In practice, for general convex \ref{mop}, we evenly sample $ \{w_{k}\}_{k\in[K]} $ from $ \mathscr{W}_{p}^{+} $ to ensure that $ S(w_{k},\theta) \in X_{P}(\theta) $. If $ \mathbf{f}(\mfx,\theta) $ is strictly convex, we can evenly sample $ \{w_{k}\}_{k\in[K]} $ from $ \mathscr{W}_{p} $ as $ S(w_{k},\theta) \in X_{P}(\theta) $ by Proposition \ref{prop:positive}.
\end{remark}

We make the following assumptions as that in \cite{dong2020imop}, which are common in inverse optimization.
\begin{assumption}
	\label{assumption:set-assumption}
	\begin{description}
		\item[(a)] $X(\theta)$ is compact, and has a nonempty relative interior. There exists $B >0$ such that $\norm{\mfx} \leq B$ for all $\mfx \in X(\theta)$. The support $ \cY $ of the noisy decisions $ \mfy $ is contained within a ball of radius $ R $ almost surely, where $ R < \infty $. In other words, $ \bP(\norm{\mfy} \leq R) = 1 $.
		\item[(b)] Each function in $ \mathbf{f} $ is strongly convex on $ \bR^{n} $, that is for each $ l \in [p] $, $ \exists \lambda_{l} > 0 $, $ \forall \mfx, \mfy \in \bR^{n} $
		\begin{align*}
		\bigg(\nabla f_{l}(\mfy,\theta_{l}) - \nabla f_{l}(\mfx,\theta_{l})\bigg)^{T}(\mfy - \mfx) \geq \lambda_{l}\norm{\mfx - \mfy}^{2}.
		\end{align*}
	\end{description}
\end{assumption}
Regarding Assumption \ref{assumption:set-assumption}.(a),  assuming that the feasible region is closed and bounded is very common in inverse optimization. The finite support of the observations is needed since we do not hope outliers have too many impacts in our learning process. Let $ \lambda = \min_{l\in[p]}\{\lambda_{l}\} $. It follows that $ w^{T}\mff(\mfx,\theta) $ is strongly convex with parameter $ \lambda $ for each $ w \in \mathscr{W}_{p} $. Therefore, Assumptions \ref{assumption:set-assumption} (a) - (b) together ensure that $ S(w,\theta) $ is a single-valued set for each $w$ and $ \theta $.

Given observations $ \{\mfy_{i}\}_{i\in[N]} $ drawn i.i.d. according to the distribution $\bP_{\mfy}$, the inverse multiobjective optimization program is given in the following.
\begin{align}
\label{imop}
\tag*{IMOP}
\begin{array}{llll}
\min\limits_{\theta \in \Theta } & \frac{1}{N}\sum\limits_{i\in[N]}l_{K}(\mfy_{i},\theta).
\end{array}
\end{align}
Statistical properties and algorithm developments have been extensively investigated in \cite{dong2020imop}. As mentioned in \cite{dong2020imop}, \ref{imop} is a bilevel program with multiple lower level problems. One typical approach for the single level reduction of the bilevel program is by substituting the lower level problem with KKT conditions. For completeness, we reformulate \ref{imop} as follows
\begin{align*}
\begin{array}{llll}
\min\limits_{\theta, \mfx_{k},\mfu_{k},\vartheta_{i,k},z_{i,k}} & \frac{1}{N}\sum\limits_{i\in[N]}\sum\limits_{k\in[K]}\norm{\mfy_i - \vartheta_{i,k} }^{2} \vspace{1mm} \\
\text{s.t.}
& \left[\begin{array}{llll}
& \mathbf{g}(\mfx_{k}) \leq \zero, \;\; \mfu_{k} \geq \zero, \vspace{1mm} \\
& \mfu_{k}^{T}\mathbf{g}(\mfx_{k}) = 0,  \vspace{1mm} \\
& \nabla_{\mfx_{k}} w_{k}^{T}\mathbf{f}(\mfx_{k},\theta) + \mfu_{k} \cdot \nabla_{\mfx_{k}}\mathbf{g}(\mfx_{k}) = \zero, \vspace{1mm}
\end{array} \right] & \forall k\in[K], \vspace{1mm} \\
& 0 \leq \vartheta_{i,k} \leq M_{i,k}z_{i,k}, & i\in[N] , \forall k\in[K], \vspace{1mm} \\
& \mfx_{k} - M_{i,k}(1 - z_{i,k}) \leq \vartheta_{i,k} \leq \mfx_{k} & i\in[N] , \forall k\in[K], \vspace{1mm} \\
& \sum\limits_{k\in[K]}z_{k} = 1, \vspace{1mm} \\
& \theta \in \Theta, \mfx_{k} \in \bR^{n},\;\; \mfu_{k} \in \bR^{q}_{+}, \;\; \vartheta_{i,k} \in \bR^{n}, \;\; z_{i,k} \in \{0,1\}, & i\in[N] , \forall k\in[K],
\end{array}
\end{align*}
where $ \mfu_{k} $ is the vector of dual variables in the KKT condtions, and $ \vartheta_{i,k} $ and $ M_{i,k} $ are used to linearize the product of $ z_{i,k} $ and $\mfx_{k} $ for each $ i\in[N] $ and $ k \in [K] $.

\subsection{Wasserstein Ambiguity Set}
Recall that $ \mathcal{Y} \subseteq \bR^{n} $ is the observation space where the observed noisy decisions take values. Denote $ \mathscr{P}(\mathcal{Y}) $ the set of all probability distributions on $ \mathcal{Y} $. From now on, we let  the Wasserstein ambiguity set $ \mathcal{P} $ be the 1-Wasserstein ball of radius $ \epsilon $ centered at $ P_{0} $:
\begin{align}\label{wasserstein ambiguity set}
\mathcal{P} = \bB_{\epsilon}(P_{0}) := \left\{Q \in \mathscr{P}(\mathcal{Y}) : \mathcal{W}(Q,P_{0}) \leq \epsilon \right\},
\end{align}
where $ P_{0} $ is the nominal distribution on $ \mathcal{Y} $, $ \epsilon > 0 $ is the radius of the set, and $ \mathcal{W}(Q,P_{0}) $ is the wasserstein distance metric of order $ 1 $ defined as \cite{villani2008optimal,esfahani2018distributionally,gao2016distributionally}
\begin{align*}
\mathcal{W}(Q,P_{0}) = \inf_{\pi \in \Pi(Q,P_{0})}\int_{\mathcal{Y}\times\mathcal{Y}} \norm{z_{1} - z_{2}}\pi(d z_{1}, d z_{2}),
\end{align*}
where $ \Pi(Q,P_{0}) $ is the set of probability distributions on $ \mathcal{Y}\times\mathcal{Y} $ with marginals $ Q $ and $ P_{0} $.

\section{Wasserstein distributionally robust IMOP}
\label{sec:wimop}
In this section, we propose the Wasserstein distributionally robust IMOP, and show its equivalence to a semi-infinite program. Subsequently, we present an algorithm to handle the resulting reformulations, and show its convergence in finite steps. Finally, we establish the statistical performance guarantees for the distributionally robust IMOP.

Given observations $ \{\mfy_{i}\}_{i\in[N]} $ drawn i.i.d. according to the distribution $\bP_{\mfy}$, the corresponding distributionally robust program of \eqref{imop} equipped with the Wasserstein ambiguity set is constructed as follows
\begin{align}
\label{wrobust imop}
\tag*{WRO-IMOP}
\begin{array}{llll}
\min\limits_{\theta \in \Theta } \sup\limits_{Q \in \bB_{\epsilon}(\widehat{P}_{N})} \bE_{\mfy \sim Q}\left[l_{K}(\mfy,\theta)\right],
\end{array}
\end{align}
which minimizes the worst case expected loss over all the distributions in the Wasserstein ambiguity set. Here $ \bB_{\epsilon}(\widehat{P}_{N}) $ is defined in \eqref{wasserstein ambiguity set}, and $ \widehat{P}_{N} $ is the empirical distribution satisfying: $ \widehat{P}_{N}(\mfy_{i}) = 1/N, \forall i \in [N] $.
\subsection{Semi-infinite reformulations}
Problem \ref{wrobust imop} involves minimizing a supremum over infinitely many distributions, which makes it difficult to solve. In this section, we establish the reformulation of \ref{wrobust imop} into a semi-infinite program.

The performance of \ref{wrobust imop} depends on how the change of $ \theta $ affects the objective values. For $ \forall w \in \mathscr{W}_{p}, \theta_{1} \in \Theta, \theta_{2} \in \Theta $, we consider the following function
\[ h(\mfx,w,\theta_{1},\theta_{2}) = w^{T}\mff(\mfx, \theta_{1}) - w^{T}\mff(\mfx,\theta_{2}). \]

\begin{assumption}[Perturbation assumption]\label{lipschitz1}
	$ \exists \kappa >0 $, $ \forall w \in \mathscr{W}_{p} $, $ \forall \theta_{1}\neq\theta_{2} \in \Theta $, $ h(\cdot,w,\theta_{1},\theta_{2}) $ is Lipschitz continuous on $ \mathcal{Y} $:$ \forall \mfx, \mfy \in \mathcal{Y}, $
	\begin{align*}
	|h(\mfx,w,\theta_{1},\theta_{2}) - h(\mfy,w,\theta_{1},\theta_{2})| \leq \kappa\norm{\theta_{1} - \theta_{2}}\norm{\mfx - \mfy}.
	\end{align*}
\end{assumption}
Basically, this assumption requires that the objective functions will not change much when either the parameter $ \theta $ or the variable $ \mfx $ is perturbed. It actually holds in many common situations, including the multiobjective linear program (MLP) and multiobjective quadratic program (MQP). As a motivating example, we provide the value of $ \kappa $ for an MQP.
\begin{example}\label{example:lipschitz}
	Suppose that $ \mff(\mfx, \theta) =  \begin{pmatrix}
	\frac{1}{2}\mfx^{T} Q_{1}\mfx + \mfc_{1}^{T}\mfx \\
	\frac{1}{2}\mfx^{T} Q_{2}\mfx + \mfc_{2}^{T}\mfx
	\end{pmatrix}$, where $ \theta = (Q_{1},Q_{2},\mfc_{1},\mfc_{2}) $. Under Assumption \ref{assumption:set-assumption}, we know that $ \norm{\mfy} \leq R $. Then, $ h(\cdot,w,\theta_{1},\theta_{2}) $ is $ 2R\norm{\theta_{1} - \theta_{2}}$-Lipschitz continuous on $ \mathcal{Y} $. That is, we can set $ \kappa = 2R $.
\end{example}

Under the previous assumptions, we will establish several properties of the loss function $ l_{K}(\mfy,\theta) $, which are ensential for our reformulation for \ref{wrobust imop}.
\begin{lemma}\label{lemma: uniformly continuous} Under Assumptions \ref{assumption:convex_setting} - \ref{lipschitz1}, the loss function $ l_{K}(\mfy,\theta) $ has the following properties:
	\begin{itemize}
		\item[(a)] $\forall \mfy \in \mathcal{Y}, \theta \in \Theta, 0 \leq l_{K}(\mfy,\theta) \leq (B+R)^{2}$.
		\item[(b)] $ l_{K}(\mfy,\theta) $ is uniformly $ 2(B+R) $-Lipschitz continuous in $ \mfy $. That is, $ \forall \theta \in \Theta, \forall \mfy_{1},\mfy_{2} \in \mathcal{Y} $, we have
		\begin{align*}
		|l_{K}(\mfy_{1},\theta) - l_{K}(\mfy_{2},\theta)| \leq 2(B+R)\norm{\mfy_{1} - \mfy_{2}}.
		\end{align*}
		\item[(c)] $ l_{K}(\mfy,\theta) $ is uniformly $ \frac{4(B+R)\kappa}{\lambda} $-Lipschitz continuous in $ \theta $. That is, $ \forall \mfy \in \cY, \forall \theta_{1},\theta_{2} \in \Theta $, we have
		\begin{align*}
		|l_{K}(\mfy,\theta_{1}) - l_{K}(\mfy,\theta_{2})| \leq \frac{4(B+R)\kappa}{\lambda}\norm{\theta_{1} - \theta_{2}}.
		\end{align*}
	\end{itemize}
\end{lemma}
(a) and (b) of this lemma are built upon direct analysis of the loss function $ l_{K}(\mfy,\theta) $. Proof of (c) is much more involved and needs the key observation that the perturbation of $ S(w,\theta) $ due to $ \theta $ is bounded by the perturbation of $ \theta $ by applying Proposition 6.1 in \cite{bonnans1998optimization}. Proof details are given in the Appendix.

Let
\begin{align*}
\mathcal{V}:= \bigg\{ \mfv\in\bR^{N+1}: & V_1 \leq v_i\leq (m+1)V_2 - m V_1, \forall i\in[N], 0 \leq v_{N+1}\leq (V_2-V_1)/\epsilon \bigg\}.
\end{align*}
where $ V_{1} $ and $ V_{2} $ are the lower and upper bounds for the loss function $ l_{K}(\mfy,\theta) $, respectively. By part (a) of Lemma \ref{lemma: uniformly continuous}, we will set $ V_{1} = 0 $, and $ V_{2} = (B+R)^{2} $ throughout the remainder of the paper.

We are now ready to state the main result in this paper.
\begin{theorem}[Semi-infinite Reformulation]\label{theorem:tractable reformulation}
	Under Assumptions \ref{assumption:convex_setting} - \ref{lipschitz1}, \ref{wrobust imop} is equivalent to the following semi-infinite program:
	\begin{align}
	\label{imop reformulation}
	\begin{array}{llll}
	\min\limits_{\theta,\mfv} & \epsilon \cdot v_{N+1} + \frac{1}{N}\sum\limits_{i\in[N]}v_{i} \vspace{1mm}\\
	\;s.t. & \sup\limits_{\widetilde{\mfy} \in \mathcal{Y}}\left(l_{K}(\widetilde{\mfy},\theta) - v_{N+1}\cdot \norm{\widetilde{\mfy} - \mfy_{i}} \right) \leq v_{i}, & \forall   i\in[N], \vspace{1mm} \\
	& \theta \in \Theta, \mfv \in \mathcal{V}
	\end{array}
	\end{align}
\end{theorem}
\begin{proof}
	Under Assumption \ref{assumption:convex_setting}, we know that $ \Theta $ is compact. Similarly, $ \mathcal{Y} $ is also compact under Assumption \ref{assumption:set-assumption} (a). By lemma \ref{lemma: uniformly continuous} (a), $\forall \mfy \in \mathcal{Y}, \theta \in \Theta, 0 \leq l_{K}(\mfy,\theta) \leq (B+R)^{2}$, and thus $ l_{K}(\mfy,\theta) $ is bounded. In addition, by lemma \ref{lemma: uniformly continuous} (b), $ l_{K}(\mfy,\theta) $ is continuous in $ \theta $ for any $ \mfy \in \mathcal{Y} $. Finally, by Lemma \ref{lemma: uniformly continuous} (c), $ l_{K}(\mfy,\theta) $ is uniformly $ \frac{4(B+R)\kappa}{\lambda} $-Lipschitz continuous in $ \mfy $. Hence, applying Corollary 3.8 of \cite{luo2017decomposition} yields the result.
\end{proof}
\begin{remark}
	The establishment of Theorem \ref{theorem:tractable reformulation} relies on those properties of the loss function $ l_{K}(\mfy,\theta) $ stated in Lemma \ref{lemma: uniformly continuous}. Although $ l_{K}(\mfy,\theta) $ might not be convex in $ \theta $ or $ \mfy $, these properties ensure that strong (Kantorovich) duality holds for the inner problem of \ref{wrobust imop}.
\end{remark}

Next, we will discuss how to incorporate the explicit form of $ l_{K}(\mfy,\theta) $ into the constraints of \eqref{imop reformulation}. For each $ i \in [N] $, constraints in \eqref{imop reformulation} is equivalent to: $ \forall \widetilde{\mfy} \in \mathcal{Y} $,
\begin{align}
\label{reformulation of constraints}
\begin{array}{llll}
\norm{\widetilde{\mfy} - \mfx_{k}}^{2} - v_{N+1}\cdot \norm{\widetilde{\mfy} - \mfy_{i}} - v_{i} \leq Mz_{ik}, \vspace{1mm} \\
\sum\limits_{k\in[K]}z_{ik} = K - 1,
\end{array}
\end{align}
where the additional constraint $ \sum_{k\in[K]}z_{ik} = K - $1 is imposed to ensure that $ \norm{\widetilde{\mfy} - \mfx_{k}}^{2} - v_{i} - v_{N+1}\cdot \norm{\widetilde{\mfy} - \mfy_{i}} \leq 0 $ for at least one $ k \in [K]$. $ M $ is an uniform upper bound for the left-hand side of the first constraint in \eqref{reformulation of constraints}. An appropriate $ M $ could be $ (B+R)^{2} $, since $ \forall i\in[N], k \in [K] $,
\begin{align*}
& \norm{\widetilde{\mfy} - \mfx_{k}}^{2} - v_{N+1}\cdot \norm{\widetilde{\mfy} - \mfy_{i}} - v_{i} \leq \norm{\widetilde{\mfy} - \mfx_{k}}^{2} \leq (B+R)^{2}.
\end{align*}

One can verify that \eqref{reformulation of constraints} is indeed equivalent to the first set of constraints in \eqref{imop reformulation} without much effort.

\begin{remark}
	We note that the semi-infinite reformulation in Theorem \ref{theorem:tractable reformulation} might still be valid if some assumption is not satisfied. Consider, for example, one of the objective functions is known to be strongly convex and the decision makers always has a positive preference for it (i.e., the weight vector $ w \in \mathscr{W}_{p}^{+} $).
\end{remark}

\subsection{Algorithm and analysis of convergence}
Theorem \ref{theorem:tractable reformulation} shows that the Wasserstein distributionally inverse multiobjective program \ref{wrobust imop} is equivalent to the semi-infinite program \eqref{imop reformulation}. Now, any existing method for solving the general semi-infinite program can be employed to solve \eqref{imop reformulation}. In particular, we are interested in using exchange methods \cite{hettich1993semi}, since our proposed algorithm inherits the spirit of these methods when applied to solving the minmax problem. The basic idea is to approximate and relax the infinite set of constraints in \eqref{imop reformulation} with a sequence of finite sets of constraints. Iteratively, new constraints are added to the previous set of constraints to tighten the relaxation by solving a maximum constraint violation problem. This is repeated until certain stopping criterion is satisfied.

Next, we discuss how to construct the finite problem, i.e., the master problem, that is a relaxation to \eqref{imop reformulation}. 

Let $ \widetilde{\mathcal{Y}}_{i} = \{\widetilde{\mfy}_{i1},\cdots,\widetilde{\mfy}_{iJ_{i}}\} \subseteq \mathcal{Y}, \forall i \in [N] $ be a collection of finite subsets of $ \mathcal{Y} $, where each subset has $ J_{i} $ samples. Then, the associated finite problem of \eqref{imop reformulation} is
\begin{align}
\label{imop reformulation 2}
\begin{array}{llll}
\min\limits_{\theta,\mfv} & \epsilon \cdot v_{N+1} + \frac{1}{N}\sum\limits_{i\in[N]}v_{i}, \vspace{1mm}\\
s.t. & l_{K}(\widetilde{\mfy}_{ij},\theta) - v_{N+1}\cdot \norm{\widetilde{\mfy}_{ij} - \mfy_{i}} \leq v_{i}, & \forall j \in [J_{i}], i\in[N], \vspace{1mm} \\
&\theta \in \Theta, v \in \mathcal{V}.
\end{array}
\end{align}
By the same arguments for the transformation from constraints in \eqref{imop reformulation} to \eqref{reformulation of constraints}, constraints in \eqref{imop reformulation 2} are equivalent to
\begin{align*}
\begin{array}{llll}
\norm{\widetilde{\mfy}_{ij} - \mfx_{k}}^{2} - v_{N+1}\cdot \norm{\widetilde{\mfy}_{ij} - \mfy_{i}} - v_{i} \leq Mz_{ijk},\vspace{1mm} \\
\sum\limits_{k\in[K]}z_{ijk} = K - 1, \;\; \forall i\in[N], j \in [J_{i}].
\end{array}
\end{align*}
\begin{proposition}[Finite problem]
	Using the above transformation, \eqref{imop reformulation 2} can be further cast into the following finite problem with finitely many constraints:
	\begin{align}
	\label{imop reformulation 3}
	\begin{array}{llll}
	\min\limits_{\theta,\mfv,\mfx_{k},z_{ijk}} & \epsilon \cdot v_{N+1} + \frac{1}{N}\sum\limits_{i\in[N]}v_{i}, \vspace{1mm}\\
	\;s.t. & \norm{\widetilde{\mfy}_{ij} - \mfx_{k}}^{2} - v_{N+1}\cdot \norm{\widetilde{\mfy}_{ij} - \mfy_{i}} - v_{i} \leq Mz_{ijk},  \vspace{1mm} \\
	& \mfx_{k} \in S(w_{k},\theta), \vspace{1mm}\\
	& \sum\limits_{k\in[K]}z_{ijk} = K - 1, \vspace{1mm} \\
	& \theta \in \Theta, v \in \mathcal{V}, z_{ijk} \in \{0,1\}, \forall i\in[N], j \in [J_{i}], k \in [K].
	\end{array}
	\end{align}
\end{proposition}

Similar to \ref{imop}, the single level reformulation for \eqref{imop reformulation 3} by applying KKT conditions is as follows
\begin{align*}
\begin{array}{llll}
\min\limits_{\theta,\mfv,\mfx_{k},z_{ijk}} & \epsilon \cdot v_{N+1} + \frac{1}{N}\sum\limits_{i\in[N]}v_{i}, \vspace{1mm}\\
\;s.t. & \norm{\widetilde{\mfy}_{ij} - \mfx_{k}}^{2} - v_{N+1}\cdot \norm{\widetilde{\mfy}_{ij} - \mfy_{i}} - v_{i} \leq Mz_{ijk},  \vspace{1mm} \\
& \left[\begin{array}{llll}
& \mathbf{g}(\mfx_{k}) \leq \zero, \;\; \mfu_{k} \geq \zero \vspace{1mm} \\
& \mfu_{k}^{T}\mathbf{g}(\mfx_{k}) = 0 \vspace{1mm} \\
& \nabla_{\mfx_{k}} w_{k}^{T}\mathbf{f}(\mfx_{k},\theta) + \mfu_{k} \cdot \nabla_{\mfx_{k}}\mathbf{g}(\mfx_{k}) = \zero \vspace{1mm}
\end{array} \right] & \forall k\in[K] \vspace{1mm} \\
& \sum\limits_{k\in[K]}z_{ijk} = K - 1, \vspace{1mm} \\
& \theta \in \Theta, v \in \mathcal{V}, \;\;\mfu_{k} \in \bR^{q}_{+}, \;\;  z_{ijk} \in \{0,1\}, \forall i\in[N], j \in [J_{i}], k \in [K].
\end{array}
\end{align*}

At each iteration, new constraints are determined to add to the previous set of constraints in \eqref{imop reformulation 3} by solving the following \textbf{Maximum constraint violation problem:}, i.e., the subproblem: $ \forall i \in [N], $
\begin{align}\label{maximum constraint violation}
CV_{i} = \max\limits_{\widetilde{\mfy} \in \mathcal{Y}}l_{K}(\widetilde{\mfy},\widehat{\theta}) - \widehat{v}_{N+1}\cdot \norm{\widetilde{\mfy} - \mfy_{i}}  - \widehat{v}_{i}.
\end{align}
Denote $ \widetilde{\mfy}_{i} $ the optimal solution of \eqref{maximum constraint violation} for each $ i \in [N] $. Whenever we find that $ CV_{i} > 0 $, we append $ \widetilde{\mfy}_{i} $ to $ \widetilde{\mathcal{Y}}_{i} $. As a result, we tighten our approximation for the infinite set of constraints in \eqref{imop reformulation} by imposing the additional constraint $ l_{K}(\widetilde{\mfy}_{i},\widehat{\theta}) - \widehat{v}_{N+1}\cdot \norm{\widetilde{\mfy}_{i} - \mfy_{i}}  - \widehat{v}_{i} \leq 0 $ in the next iteration.

With the above assumptions and analysis, we now present our method to solve \ref{wrobust imop} in Algorithm \ref{alg:algorithm1} shown in Figure \ref{fig:scheme}.
\begin{algorithm}
	\caption{Wasserstein Distributionally Robust IMOP}
	\label{alg:algorithm1}
	\begin{algorithmic}[1]
		\STATE{\bfseries Input:} noisy decisions $\{\mfy_{i}\}_{i \in [N]}$, weights $\{w_{k}\}_{k\in K}$, radius $ \epsilon $ of Wasserstein ball, and stopping tolerance $ \delta $
		\STATE{\bfseries Initialize} $\widetilde{\mathcal{Y}}_i\gets\emptyset, \forall i \in [N]$
		\REPEAT
		\STATE solve the master problem in \eqref{imop reformulation 3} with $ \widetilde{\mathcal{Y}}_{i}, \forall i \in [N] $, and return an optimal solution $ (\widehat{\theta},\widehat{\mfv}) $
		\FOR{$ i=1,\ldots,N $}
		\STATE solve the subproblem, i.e., the maximum constraint violation problem \eqref{maximum constraint violation}
		\STATE \textbf{if} {$ CV_i > 0 $} \textbf{then} let $ \widetilde{\mathcal{Y}}_{i} \leftarrow \widetilde{\mathcal{Y}}_{i} \cup \{\widetilde{\mfy}_{i}\} $ \textbf{end if}
		\ENDFOR
		\UNTIL{$ \max_{i\in[N]}CV_{i} \leq \delta $}
		\STATE{\bfseries Output:} a $ \delta$-optimal solution $ \widehat{\theta}_{N} $ of \eqref{imop reformulation}
	\end{algorithmic}
\end{algorithm}
\begin{remark}
	In Step 6, the maximum constraint violation problem can be solved exactly and efficiently by invoking solver such as Baron \cite{sahinidis1996baron}. Nevertheless, it can also be solved  by decomposing into $ K $ subproblems through substituting $ \{\mfx_{k}\}_{k\in[K]} $ into $l_{K}(\widetilde{\mfy},\widehat{\theta})$. These $ K $ different subproblems can be solved independently and in parallel, allowing a linear speedup of Step 6. 
\end{remark}
\begin{figure}[ht]
	\centering
	\includegraphics[width=0.9\linewidth]{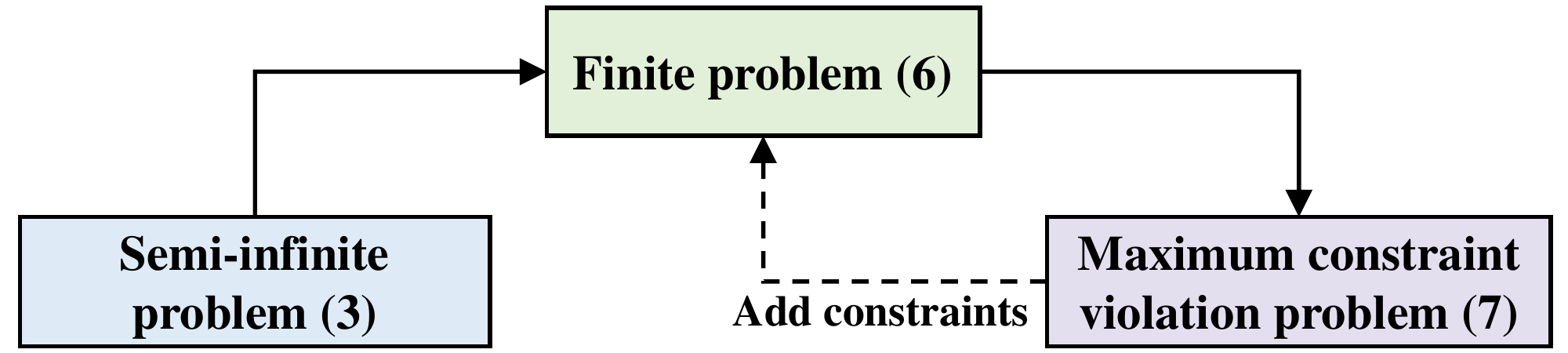}
	\caption{General scheme of Algorithm \ref{alg:algorithm1}.}
	\label{fig:scheme}
\end{figure}

For completeness, we give the convergence proof of Algorithm \ref{alg:algorithm1} in the following theorem.
\begin{theorem}[Finite convergence]\label{theorem:convergence}
	Under Assumptions \ref{assumption:convex_setting} - \ref{lipschitz1}, Algorithm \ref{alg:algorithm1} converges within $ (\frac{GR_{0}}{\delta} + 1)^{n_{\theta} + N + 1} $ iterations. Here,
	\begin{align*}
	G  & = (1 + 2R +  \frac{4(B+R)\kappa}{\lambda}),\\
	R_{0} & = \sqrt{D^{2} + N\big((m+1)V_2 - m V_1\big)^{2} + \left(\frac{V_2-V_1}{\epsilon}\right)^{2}}.
	\end{align*}
\end{theorem}
\begin{proof}
	For ease of notation, we denote $ (\widehat{\theta}^{s},\widehat{\mfv}^{s}) $ the solution found in Step 4  in the $ s $th iteration of Algorithm 1.
	
	Suppose that for $ s = 1,\ldots,S $ the algorithm has not terminated, i.e. $ \max\limits_{i\in[N]}l_{K}(\widetilde{\mfy}_{i},\widehat{\theta}) - \widehat{v}_{N+1}\cdot \norm{\widetilde{\mfy}_{i} - \mfy_{i}} - \widehat{v}_{i} > \delta $.
	
	Let $ i^{*} = \argmax\limits_{i\in[N]}l_{K}(\widetilde{\mfy}_{i},\widehat{\theta}) - \widehat{v}_{N+1}\cdot \norm{\widetilde{\mfy}_{i} - \mfy_{i}} - \widehat{v}_{i} $. Then, $ \widetilde{\mfy}_{i^{*}} $ is added to $ \widetilde{\mathcal{Y}}_{i^{*}} $. We have
	\begin{align}
	\label{proof-convergence 1}
	l_{K}(\widetilde{\mfy}_{i^{*}},\widehat{\theta}^{s}) - \widehat{v}_{N+1}^{s}\cdot \norm{\widetilde{\mfy}_{i^{*}} - \mfy_{i^{*}}} - \widehat{v}_{i^{*}}^{s} > \delta.
	\end{align}
	$ \forall t > s $, we know that
	\begin{align}
	\label{proof-convergence 2}
	l_{K}(\widetilde{\mfy}_{i^{*}},\widehat{\theta}^{t}) - \widehat{v}_{N+1}^{t}\cdot \norm{\widetilde{\mfy}_{i^{*}} - \mfy_{i^{*}}} - \widehat{v}_{i^{*}}^{t} \leq 0.
	\end{align}
	
	Combining \eqref{proof-convergence 1} and \eqref{proof-convergence 2}, we have that, for $ s < t $,
	\begin{align}
	\label{proof-convergence 3}
	l_{K}(\widetilde{\mfy}_{i^{*}},\widehat{\theta}^{s}) - \widehat{v}_{N+1}^{s}\cdot \norm{\widetilde{\mfy}_{i^{*}} - \mfy_{i^{*}}} - \widehat{v}_{i^{*}}^{s} - (l_{K}(\widetilde{\mfy}_{i^{*}},\widehat{\theta}^{t}) - \widehat{v}_{N+1}^{t}\cdot \norm{\widetilde{\mfy}_{i^{*}} - \mfy_{i^{*}}} - \widehat{v}_{i^{*}}^{t}) > \delta.
	\end{align}
	Note that
	\begin{align}
	\label{proof-convergence 4}
	\begin{array}{llll}
	& l_{K}(\widetilde{\mfy}_{i^{*}},\widehat{\theta}^{s}) - \widehat{v}_{N+1}^{s}\cdot \norm{\widetilde{\mfy}_{i^{*}} - \mfy_{i^{*}}} - \widehat{v}_{i^{*}}^{s} - (l_{K}(\widetilde{\mfy}_{i^{*}},\widehat{\theta}^{t}) - \widehat{v}_{N+1}^{t}\cdot \norm{\widetilde{\mfy}_{i^{*}} - \mfy_{i^{*}}} - \widehat{v}_{i^{*}}^{t})\\
	&= l_{K}(\widetilde{\mfy}_{i^{*}},\widehat{\theta}^{s}) - l_{K}(\widetilde{\mfy}_{i^{*}},\widehat{\theta}^{t}) - (\widehat{v}_{N+1}^{s}\cdot \norm{\widetilde{\mfy}_{i^{*}} - \mfy_{i^{*}}} - \widehat{v}_{N+1}^{t}\cdot \norm{\widetilde{\mfy}_{i^{*}} - \mfy_{i^{*}}}) - (\widehat{v}_{i^{*}}^{s} - \widehat{v}_{i^{*}}^{t}) \\
	&\leq \frac{4(B+R)\kappa}{\lambda}\norm{\widehat{\theta}^{s} - \widehat{\theta}^{t}} + 2R|\widehat{v}_{N+1}^{s} - \widehat{v}_{N+1}^{t}| + |\widehat{v}_{i^{*}}^{s} - \widehat{v}_{i^{*}}^{t}| \\
	&\leq (1 + 2R +  \frac{4(B+R)\kappa}{\lambda})\norm{(\widehat{\theta}^{s},\widehat{\mfv}^{s}) - (\widehat{\theta}^{t},\widehat{\mfv}^{t})}
	\end{array}
	\end{align}
	where the first inequality is due to the Lipschitz condition in Lemma \ref{lemma: uniformly continuous}(c).
	
	Combining \eqref{proof-convergence 3} and \eqref{proof-convergence 4}, we have that, for $ s < t $,
	\begin{align}
	\label{proof-convergence 5}
	\norm{(\widehat{\theta}^{s},\widehat{\mfv}^{s}) - (\widehat{\theta}^{t},\widehat{\mfv}^{t})} > \frac{\delta}{G}.
	\end{align}
	Thus, the minimum distance between any two solutions $ (\widehat{\theta}^{1},\widehat{\mfv}^{1}), \ldots,(\widehat{\theta}^{S},\widehat{\mfv}^{S}) $ exceeds $ \delta/G $.
	
	Let $ \mathscr{B}_{s} \in \bR^{n_{\theta} + N + 1} $ denote the ball centered at $ (\widehat{\theta}^{s},\widehat{\mfv}^{s}) $ with radius $ \delta/G $. Let $ R_{0} = \sqrt{D^{2} + N((m+1)V_2 - m V_1)^{2} + (\frac{(V_2-V_1)}{\epsilon})^{2}} $. Let $ \mathscr{B} $ denote the ball centered at origin with radius $ R_{0} + \frac{\delta}{G} $. It follows that the balls $ \{ \mathscr{B}_{s}\}_{s\in[S]} $, which do not intersect with each other, are covered by $ \mathscr{B} $. Thus, we have
	\begin{align*}
	S\beta_{n_{\theta} + N + 1}(\frac{\delta}{G})^{n_{\theta} + N + 1} \leq \beta_{n_{\theta} + N + 1}(R_{0} + \frac{\delta}{G})^{n_{\theta} + N + 1}
	\end{align*}
	where $ \beta_{n_{\theta} + N + 1} $ is the volume of the unit ball in $ \bR^{n_{\theta} + N + 1} $. Thus, we conclude that
	\begin{align*}
	S \leq (\frac{GR_{0}}{\delta} + 1)^{n_{\theta} + N + 1}.
	\end{align*}
\end{proof}
\begin{remark}
	The proof of convergence is in spirit similar to that of the cutting plane methods for robust optimization and distributionally robust optimization \cite{joachims2009cutting,mutapcic2009cutting,luo2017decomposition}. In practice, we mention that the actual number of iterations typically required is much smaller than $ (\frac{GR_{0}}{\delta} + 1)^{n_{\theta} + N + 1} $.
\end{remark}

\begin{remark}[Variants of Algorithm \ref{alg:algorithm1}]
	Note that we add $ \widetilde{\mfy}_{i} $ to $\widetilde{\mathcal{Y}}_{i} $ whenever $ V_{i} > 0 $ for each $ i \in [N] $. Nevertheless, from the convergence proof, it suffices to add only one  $ \widetilde{\mfy}_{i} $ corresponding to the biggest $ V_{i} $ that are positive. Consequently, we dramatically ease the computational burden in each iteration.
\end{remark}

\subsection{Performance guarantees}
One of the main goals of statistical analysis of learning algorithms is to understand how the excess risk of a data dependent decision rule output by the empirical risk minimization depends on the sample size of the observations and on the "complexity" of the class $ \Theta $. Next, we provide a performance guarantee for \ref{wrobust imop} by showing below that the excess risk of the estimator obtained by solving \ref{wrobust imop} would converge sub-linearly to zero.
\begin{theorem}[Excess risk bound]\label{theorem:excess risk bound}
	Define the minimax risk estimator
	\begin{align*}
	\theta^{*} \in \argmin_{\theta \in \Theta}\left\{ \sup\limits_{Q \in \bB_{\epsilon}(P)} \bE_{\mfy \sim Q}\left[l_{K}(\mfy,\theta)\right]\right\},
	\end{align*}
	where $ P $ is the distribution from which the observations $ \{\mfy_{i}\}_{i\in[N]} $ are drawn, and the minimax empirical risk estimator
	\begin{align*}
	\widehat{\theta}_{N} \in \argmin_{\theta \in \Theta}\left\{ \sup\limits_{Q \in \bB_{\epsilon}(\widehat{P}_{N})} \bE_{\mfy \sim Q}\left[l_{K}(\mfy,\theta)\right]\right\}.
	\end{align*}
	and $ \widehat{P}_{N} $ is the empirical distribution of the observations $ \{\mfy_{i}\}_{i \in [N]} $.
	
	Under Assumptions \ref{assumption:convex_setting} - \ref{lipschitz1}, $ \forall 0 <\delta <1 $, the following holds with probability at least $ 1 - \delta $:
	\begin{align*}
	\begin{array}{llll}
	\sup\limits_{Q \in \bB_{\epsilon}(P)} &\bE_{\mfy \sim Q}\left[l_{K}(\mfy,\widehat{\theta}_{N})\right]  - \sup\limits_{Q \in \bB_{\epsilon}(P)} \bE_{\mfy \sim Q}\left[l_{K}(\mfy,\theta^{*})\right] \leq \frac{H}{\sqrt{N}} + \frac{3(B+R)^2\sqrt{\log(2/\delta)}}{\sqrt{2N}}
	\end{array}
	\end{align*}
	where $H$ is a constant depending only on $D, B, R, n_{\theta},\kappa$:
	\begin{align*}
	H &=  96\left(\frac{3D\sqrt{n_{\theta}}}{\kappa} + 2R\right)(B+R).
	\end{align*}
\end{theorem}
\begin{proof}
	By Lemma \ref{lemma: uniformly continuous}(a), we have $\mathcal{Y})$ is bounded since $ diam(\mathcal{Y}) \leq 2R$, where $ diam(\mathcal{Y}) $ is the diameter of the observation space $ \mathcal{Y} $. $ \forall \theta \in \Theta$, the loss function $ l_{K}(\mfy,\theta) $ is $ 2(B+R) $-Lipschitz continuous in $ \mfy $, and $ 0 \leq l_{K}(\mfy,\theta) \leq (B+R)^{2} $ by Lemma \ref{lemma: uniformly continuous}  (a) and (b). We next extends Theorem 2 in \cite{NIPS2018_7534} to construct our proof of the excess risk bound.

	Denote $ \mathcal{F}:= \{l_{K}(\cdot,\theta):\theta \in \Theta \} $ the class of the loss functions. Before evaluating the Dudley entropy integral, we need to estimate the covering number $ \mathcal{N}(\mathcal{F},\norm[\infty]{\cdot},\cdot) $. First observe that for any two $ l_{K}(\cdot,\theta_{1}), l_{K}(\cdot,\theta_{2}) \in \mathcal{F} $ corresponding to $ \theta_{1},\theta_{2} \in \Theta $, we have
	\begin{align*}
	|l_{K}(\mfy,\theta_{1}) - l_{K}(\mfy,\theta_{2})| \leq \frac{4(B+R)\kappa}{\lambda}\norm{\theta_{1} - \theta_{2}}.
	\end{align*}
	Since $ \Theta $ belongs to the ball in $\bR^{n_{\theta}}$ with radius $ D $,
	\begin{align*}
	\mathcal{N}(\mathcal{F},\norm[\infty]{\cdot},u)  \leq \mathcal{N}(\Theta,\norm[2]{\cdot},\frac{\lambda}{4(B+R)\kappa}u) \leq \left(\frac{12D(B+R)}{\kappa u}\right)^{n_{\theta}}
	\end{align*}
	for $0 < u < \frac{4D(B+R)}{\kappa}$, and $\mathcal{N}(\mathcal{F},\norm[\infty]{\cdot},u) = 1$ for $u \geq \frac{4D(B+R)}{\kappa}$, which leads to
	\begin{align*}
	\int^\infty_0 \sqrt{\log \mathcal{N}(\mathcal{F},\norm[\infty]{\cdot},u) } d u &\leq \int^{\frac{4D(B+R)}{\kappa}}_0 \sqrt{n_{\theta}\log \left(\frac{12D(B+R)}{\kappa u}\right)} d u \\
	& = \frac{12D(B+R)}{\kappa}\sqrt{n_{\theta}}\int^{1/3}_0 \sqrt{\log \left(1/u\right)} d u \\
	& \leq \frac{6D(B+R)}{\kappa}\sqrt{n_{\theta}}.
	\end{align*}
	Substituting the above inequality into the bound provided in Theorem 2 of \cite{NIPS2018_7534}, we get the desired estimate.
\end{proof}
\begin{remark}[Performance Guarantees]
	\begin{itemize}
		\item The bounded support assumption for $ \cY $ of the noisy decisions is restrictive but seems to be unavoidable for any a priori guarantees of the type described in Theorem \ref{theorem:excess risk bound}. In future work, we will investigate whether we could obtain other types of performace guarantees while relaxing $\bP_{\mfy}$ to be light-tailed.
		\item Analogous to the convergence rate of empirical risk minimization when $ \epsilon = 0 $, we get an $ \mathcal{O}(1/\sqrt{N}) $ excess risk bound. However, the obtained excess risk bound does not depend on the radius $ \epsilon $ of the Wasserstein ambiguity set. Similar to \cite{NIPS2018_7534}, this phenomenon is due to the fact that we are using the Lipschitz continuity of the loss function $ l_{K}(\mfy,\theta) $.
		\item  The right terms in the excess risk bound inequality increase as either $D, B, R, n_{\theta}$ grow or $ \kappa $ shrinks, indicating that the learnability of the decision making model decreases. This is consistent with our observation that uncertainties in the model, data, and parameter space will increase the difficulty of learning the parameters through inverse multiobjective optimization in general.
	\end{itemize}
\end{remark}

\section{Experiments}
\label{sec:experiments wimop}
In this section, we will provide an MQP and a portfolio optimization problem to illustrate the performance of the proposed algorithm \ref{alg:algorithm1}. The mixed integer second order conic problems are solved by Gurobi \cite{gurobi}. All the algorithms are programmed with Julia \cite{bezanson2017julia}. The single level reformulations for \ref{imop} and \ref{wrobust imop} are provided in \ref{Appendix: single level reformulation of MQP} and \ref{Appendix: single level reformulation of wro-imop with MQP}.

\subsection{Sythetic data: learning the objective functions of an MQP}
Consider the following multiobjective quadratic optimization problem.
\begin{align*}
\begin{array}{llll}
\min\limits_{\mfx \in \bR_{+}^{2}} & \left( \begin{matrix} f_{1}(\mfx) = \frac{1}{2}\mfx^{T}Q_{1}\mfx + \mfc_{1}^{T}\mfx  \\ f_{2}(\mfx) = \frac{1}{2}\mfx^{T}Q_{2}\mfx + \mfc_{2}^{T}\mfx\end{matrix} \right) \vspace{1mm} \\
\;s.t.   & A\mfx \leq \mfb
\end{array}
\end{align*}
where the parameters of the two objective functions are
\begin{align*}
Q_{1} = \begin{bmatrix}
1 & 0 \\
0 & 2
\end{bmatrix}, \mfc_{1} = \begin{bmatrix}
-0.5 \\
-1
\end{bmatrix},
Q_{2} = \begin{bmatrix}
2 & 0 \\
0 & 1
\end{bmatrix}, \mfc_{2} = \begin{bmatrix}
-5 \\
-2.5
\end{bmatrix},
\end{align*}
and the parameters for the feasible region are
\begin{align*}
A = \begin{bmatrix}
1 & 0 \\
0 & 1
\end{bmatrix}, \mfb = \begin{bmatrix}
3 \\
3
\end{bmatrix}.
\end{align*}

\begin{figure*}
	\centering
	\begin{subfigure}[t]{0.45\textwidth}
		\includegraphics[width=\textwidth]{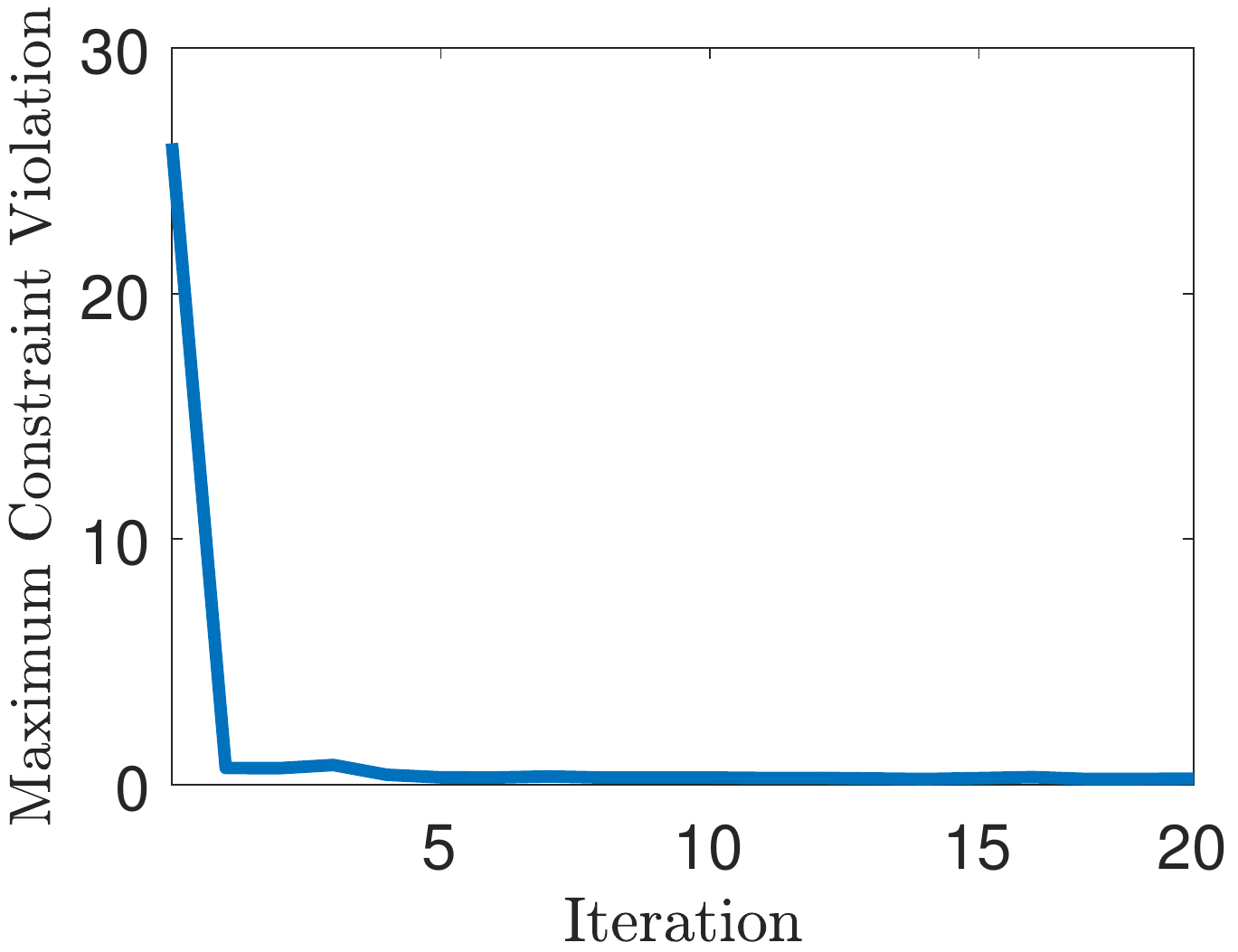}
		\caption{}
		\label{fig:qp_c_criteria}
	\end{subfigure}
	\hspace{20pt}
	\begin{subfigure}[t]{0.45\textwidth}
		\includegraphics[width=\textwidth]{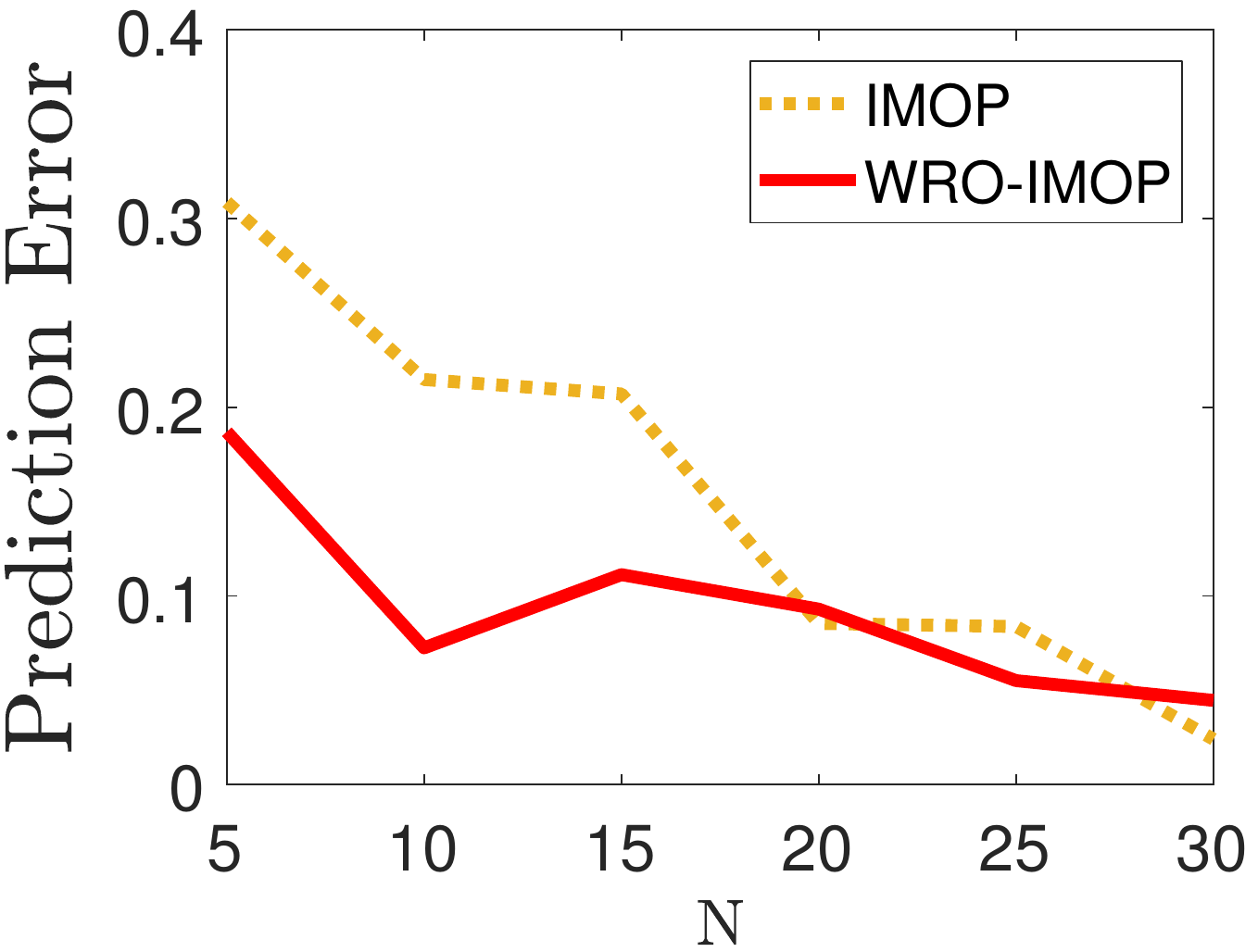}
		\caption{}
		\label{fig:qp_c_compare}
	\end{subfigure}
	\caption{Learning the objective functions of a Multiobjective quadratic program. Results are averaged over 10 repetitions. (a) Maximum constraint violation versus iteration for $ N = 15 $. (b) Prediction errors for two methods with different $ N $.
	}
	\label{fig:qp_c_online}
\end{figure*}

We seek to learn $ \mfc_{1}$ and  $\mfc_{2}$ in this experiment. The data is generated as follows. We first compute Pareto optimal solutions $ \{\mfx_{i}\}_{i \in [N]} $ by solving \ref{weighting problem} with weight samples $ \{w_{i}\}_{i \in [N]} $ that are uniformly chosen from $ \mathscr{W}_{2} $. Next, the noisy decision $ \mfy_{i} $ is obtained by adding noise to $ \mfx_{i} $ for each $ i \in [N] $. More precisely, $ \mfy_{i} = \mfx_{i} + \epsilon_{i} $, where each element of $ \epsilon_{i} $ has a uniform distribution supporting on $ [-0.25,0.25] $ with mean $ 0 $ for all $ i \in [N] $.

We assume that $\mfc_{1}$ and $ \mfc_{2} $ are within $[-6,0]^{2}$, and the first elements for them are given. $K = 6$ weights from $\mathscr{W}_{2}$ are evenly sampled. The radius $ \epsilon $ of the Wasserstein ambiguity set is selected from the set $ \{10^{-4},10^{-3},10^{-2},10^{-1},1\} $. We report below the results with the lowest prediction error across all candidate radii. The stopping criteria $ \delta $ is set to be $ 0.1 $. Then, we implement Algorithm \ref{alg:algorithm1} with different $ N $.

To illustrate the performance of the algorithm in a statistical way, we run $10$ repetitions of the experiments. Figure \ref{fig:qp_c_criteria} shows the maximum constraint violation $ \max_{i\in[N]}V_{i} $ versus iteration for one repetition when $ N = 10 $. As can be seen in the figure, the algorithm converges very fast. In Figure \ref{fig:qp_c_compare}, we report the prediction errors averaged over $ 10 $ repetitions with both the robust and non-robust approaches for different $ N $. Here, we use an independent validation set that consists of $ 10^5 $ noisy decisions generated in the same way as the training data to compute the prediction error. The experiments suggest that the Wasserstein distributionally robust approach can significnatly reduce the prediction error, especially when $ N $ is small, i.e., we have a very limited number of observations.

\subsection{Real world case study: learning the expected returns}

In this example, we consider various noisy decisions arising from different investors in a stock market. More precisely, we consider a portfolio selection problem, where investors need to determine the fraction of their wealth to invest in each security in order to maximize the total return and minimize the total risk. The classical Markovitz mean-variance portfolio selection \cite{markowitz1952portfolio} in the following is frequently employed by analysts.
\begin{align*}
\begin{array}{llll}
\min & \left( \begin{matrix}f_{1}(\mfx) &= -\mathbf{r}^{T}\mfx \\ f_{2}(\mfx) &= \mfx^{T}Q\mfx\end{matrix} \right) \vspace{1mm} \\
\;s.t.   & 0 \leq x_{i} \leq b_{i} &\forall i \in [n] \\
& \sum\limits_{i=1}^{n}x_{i} = 1
\end{array}
\end{align*}
where  $ \mathbf{r} \in \bR^{n}_{+}$  is a vector of individual security expected returns, $Q \in \bR^{n \times n}$ is the covariance matrix of securities returns, $\mfx$ is a portfolio specifying the proportions of capital to be invested in the different securities, and $b_{i}$ is an upper bound on the proportion of security $i$, $\forall i \in [n]$.

\textbf{Dataset}: The dataset is derived from monthly total returns of 30  stocks from a blue-chip index which tracks the performance of top 30 stocks in the market when the total investment universe consists of thousands of assets. The true expected returns and true return covariance  matrix for the first $8$ securities are given in the supplementary material. Suppose a learner seeks to learn the expected return for the first four securities that an analyst uses based on $20$ noisy decisions from investors served by this analyst.

The noisy decision for each investor $i \in [20]$ is generated as follows. We set each upper bound for the proportion of the $8$ securities to $b_{i} = 1.0 , \forall i \in [8]$. Then, we uniformly sample $20$ weights and use them to generate optimal portfolios on the efficient frontier that is plot in Figure \ref{fig:effcientfrontierport}. Subsequently, each component of  these portfolios is rounded to the nearest thousandth, which can be seen as measurement error. The radius $ \epsilon $ of the Wasserstein ambiguity set is selected from the set $ \{10^{-4},10^{-3},10^{-2},10^{-1},1\} $. The stopping criteria $ \delta $ is set to be $ 0.1 $.

\begin{figure*}
	\centering
	\begin{subfigure}[t]{0.45\textwidth}
		\includegraphics[width=1\linewidth]{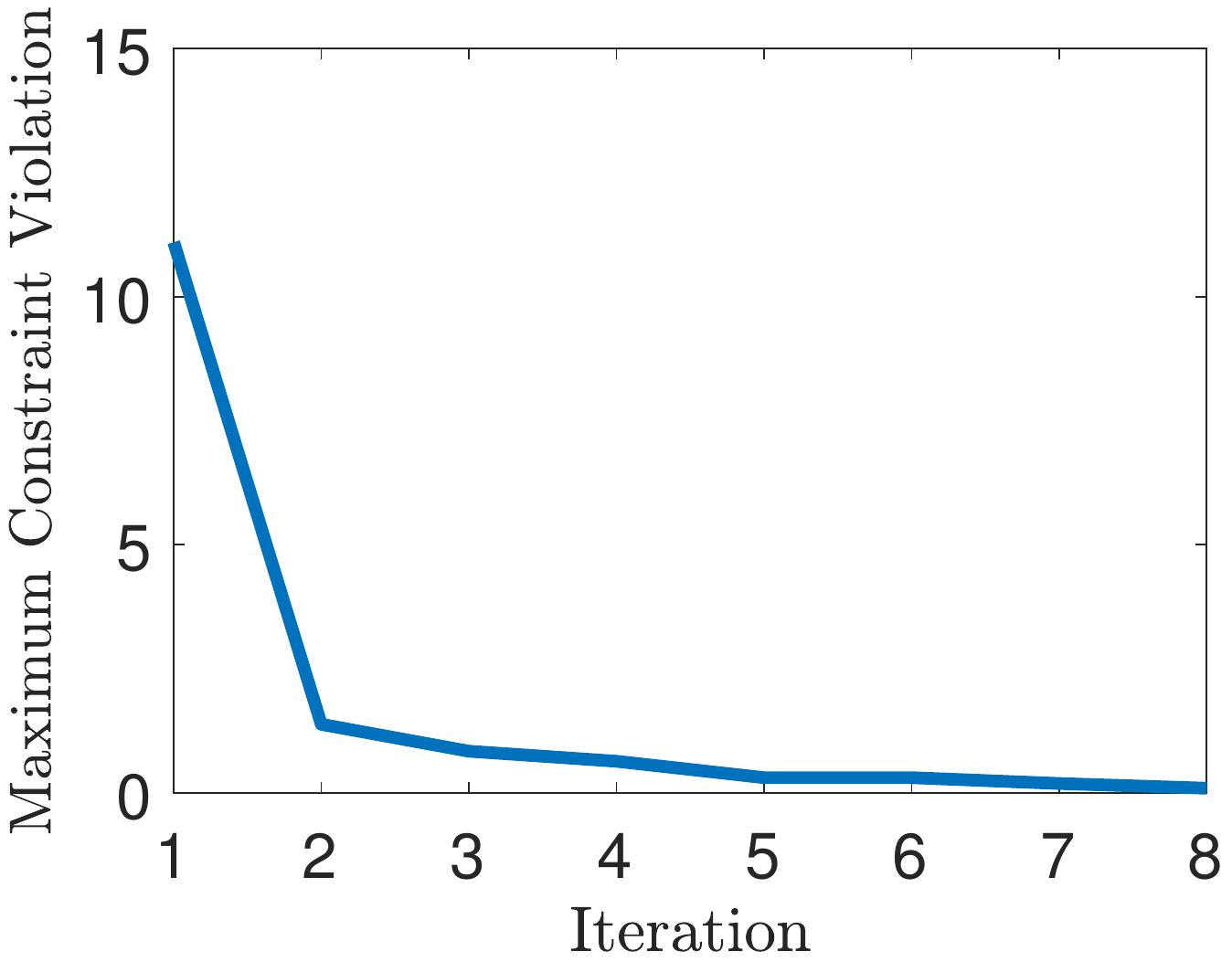}
		\caption{}
		\label{fig:portmax}
	\end{subfigure}
	\hspace{20pt}
	\begin{subfigure}[t]{0.43\textwidth}
		\includegraphics[width=1\linewidth]{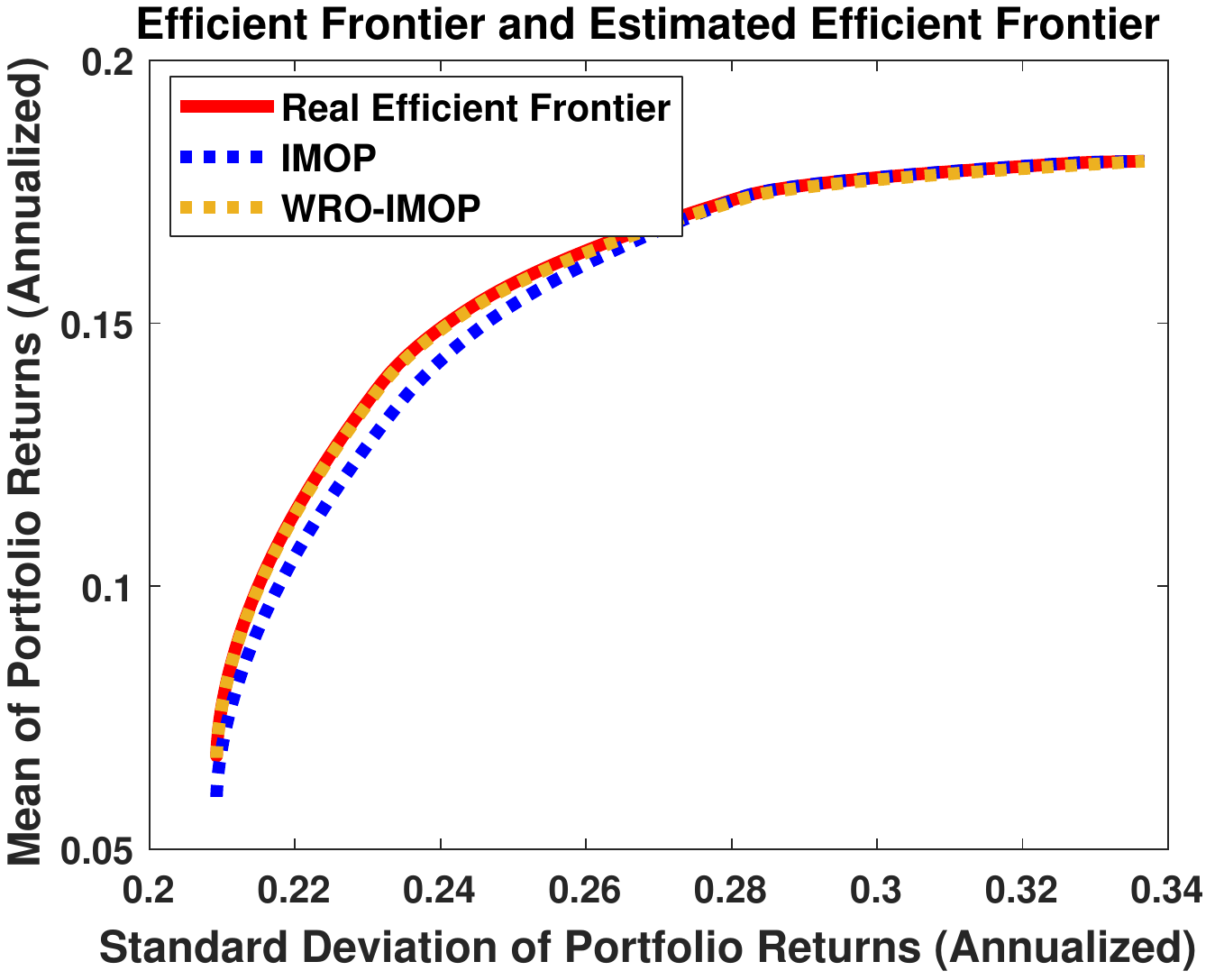}
		\caption{}
		\label{fig:effcientfrontierport}
	\end{subfigure}
	\caption{Learning the expected return. (a) Maximum constraint violation versus iteration. (b) The red line indicates the real efficient frontier. The yellow dots indicates the estimated efficient frontier using the distributionally robust approach. The blue dots indicates the estimated efficient frontier using the non-robust approach.}
	\label{fig:port wimop}
\end{figure*}

Figure \ref{fig:portmax} shows that our algorithm converges in $ 8 $ iterations. We also plot the estimated efficient frontiers using both the robust and non-robust approaches with $K = 6$ in Figure \ref{fig:effcientfrontierport}. We can see that the estimated efficient frontier of the Wasserstein distributionally robust approach is closer to the real one than the non robust approach, showing that our method in this paper allows for a lower prediction error when only limited number of decisions observed are accessible. Note that the first function is not strongly convex. The experiment results suggest that our methodology is generalizable to a broader class of problems.

\section{Conclusions}

In this paper, we present a novel Wasserstein distributionally robust framework for constructing inverse multiobjective optimization estimator. We show that the proposed framework has statistical performance guarantees, and the excess risk of the distributionally robust inverse multiobjective optimization estimator would converge to zero with a sub-linear rate as the number of observed decisions approaches to infinity. To solve the resulting minmax problem, we reformulate it as a semi-infinite program and develop a cutting-plane algorithm which converges to a $\delta$-optimal solution in finite iterations. We demonstrate the effectiveness of our method on both a multiobjective quadratic program and a portfolio optimization problem.


%
%
%
%
%

\newpage
\appendix
\numberwithin{equation}{section}

\section{Omitted mathematical reformulations}
Before giving the reformulations, we first make some discussions about the surrogate loss functions.
\begin{align*}
l_{K}(\mfy,\theta) & = \min_{z_{k} \in \{0,1\}}\norm{\mfy - \sum_{k\in[K]}z_{k}\mfx_{k}}^{2} \vspace{1mm}\\
& = \min_{z_{k} \in \{0,1\}}\sum_{k\in[K]}\norm{\mfy - z_{k}\mfx_{k}}^{2} - (K - 1)\norm{\mfy}^{2} \vspace{1mm}\\
\end{align*}
where $ \mfx_{k} \in S(w_{k},\theta) $ and $ \sum_{k\in[K]}z_{k} = 1 $.

Since $ (K - 1)\norm{\mfy}^{2} $ is a constant, we can safely drop it and use the following as the surrogate loss function when solving the optimization program in the implicit update,
\begin{align*}
l_{K}(\mfy,\theta) = \min_{z_{k} \in \{0,1\}}\sum_{k\in[K]}\norm{\mfy - z_{k}\mfx_{k}}^{2}
\end{align*}
where $ \mfx_{k} \in S(w_{k},\theta) $ and $ \sum_{k\in[K]}z_{k} = 1 $.

%
%

\subsection{Single level reformulation for the Inverse Multiobjective Quadratic Problem}\label{Appendix: single level reformulation of MQP}
When the objective functions are quadratic and the feasible region is a polyhedron, the multiobjective optimization has the following form
\begin{align*}
\label{mqp}
\tag*{MQP}
\begin{array}{llll}
\min\limits_{\mfx \in \bR^{n}} & \begin{bmatrix}
\frac{1}{2}\mfx^{T} Q_{1}\mfx + \mfc_{1}^{T}\mfx \\
\vdots \\
\frac{1}{2}\mfx^{T} Q_{p}\mfx + \mfc_{p}^{T}\mfx
\end{bmatrix} \vspace{2mm}\\
\;s.t. & A\mfx \geq \mfb
\end{array}
\end{align*}
where $ Q_{l} \in \mathbf{S}^{n}_{+} $ (the set of symmetric positive semidefinite matrices) for all $ l \in [p] $..

When trying to learn $ \{\mfc_{l}\}_{l \in [p]} $, the single level reformulation for the \ref{imop} is given in the following 	
\begin{align*}
\begin{array}{llll}
\min\limits_{\theta, \mfx_{k},\mfu_{k},\vartheta_{i,k},z_{i,k}} & \frac{1}{N}\sum\limits_{i\in[N]}\sum\limits_{k\in[K]}\norm{\mfy_i - \vartheta_{i,k} }^{2} \vspace{1mm} \\
\text{s.t.}
& \left[\begin{array}{llll}
& \mathbf{g}(\mfx_{k}) \leq \zero, \;\; \mfu_{k} \geq \zero, \vspace{1mm} \\
& \mfu_{k}^{T}\mathbf{g}(\mfx_{k}) = 0,  \vspace{1mm} \\
& \nabla_{\mfx_{k}} w_{k}^{T}\mathbf{f}(\mfx_{k},\theta) + \mfu_{k} \cdot \nabla_{\mfx_{k}}\mathbf{g}(\mfx_{k}) = \zero, \vspace{1mm}
\end{array} \right] & \forall k\in[K], \vspace{1mm} \\
& 0 \leq \vartheta_{i,k} \leq M_{i,k}z_{i,k}, & i\in[N] , \forall k\in[K], \vspace{1mm} \\
& \mfx_{k} - M_{i,k}(1 - z_{i,k}) \leq \vartheta_{i,k} \leq \mfx_{k} & i\in[N] , \forall k\in[K], \vspace{1mm} \\
& \sum\limits_{k\in[K]}z_{k} = 1, \vspace{1mm} \\
& \theta \in \Theta, \mfx_{k} \in \bR^{n},\;\; \mfu_{k} \in \bR^{q}_{+}, \;\; \vartheta_{i,k} \in \bR^{n}, \;\; z_{i,k} \in \{0,1\}, & i\in[N] , \forall k\in[K],
\end{array}
\end{align*}
\begin{align*}
\begin{array}{llll}
\min\limits_{\mfc_{l}, \mfx_{k},\mfu_{k},\mft_{k} ,\vartheta_{i,k},z_{i,k}} & \frac{1}{N}\sum\limits_{i\in[N]}\sum\limits_{k\in[K]}\norm{\mfy_i - \vartheta_{i,k} }^{2} \vspace{1mm} \\
\text{s.t.} & \mfc_{l} \in \widetilde{C}_{l} & \forall l \in [p]\vspace{1mm} \\
& \left[\begin{array}{llll}
& \mfA\mfx_{k} \geq \mfb,\; \mfu_{k} \geq \zero \vspace{1mm}\\
& \mfu_{k} \leq M\mft_{k} \vspace{1mm} \\
& \mfA\mfx_{k} - \mfb \leq M(1 - \mft_{k})  \vspace{1mm} \\
& (w_{k}^{1}Q_{1} + \cdots + w_{k}^{p}Q_{p})\mfx_{i} + w_{k}^{1}\mfc_{1} + \cdots + w_{k}^{p}\mfc_{p} - \mfA^{T}\mfu_{k} = 0 \vspace{1mm}
\end{array} \right] & \forall k\in[K] \vspace{1mm} \\
& 0 \leq \vartheta_{i,k} \leq M_{i,k}z_{i,k}, & i\in[N] , \forall k\in[K], \vspace{1mm} \\
& \mfx_{k} - M_{i,k}(1 - z_{i,k}) \leq \vartheta_{i,k} \leq \mfx_{k} & \forall i\in[N] , \forall k\in[K], \vspace{1mm} \\
& \sum\limits_{k\in[K]}z_{k} = 1 \vspace{1mm} \\
& \mfc_{l} \in \widetilde{C}_{l}, \;\; \mfx_{k} \in \bR^{n},\;\; \mfu_{k} \in \bR^{q}_{+}, \;\; \mft_{k} \in \{0,1\}^{q}, \;\; \vartheta_{i,k} \in \bR^{n}, \;\;z_{i,k} \in \{0,1\} & \forall l \in [p],\forall i\in[N],\forall k\in[K]
\end{array}
\end{align*}
where  $ \widetilde{C}_{l} $ is a convex set for each $ l \in [p] $.  We have a similar single level reformulation when learning the Right-hand side $ \mfb $.

%
%

\subsection{Single level reformulation for the Finite problem \eqref{imop reformulation 3}  with Multiobjective Quadratic Problem} \label{Appendix: single level reformulation of wro-imop with MQP}
\begin{align*}
\begin{array}{llll}
\min\limits_{\theta,\mfv,\mfx_{k},\mfu_{k},\mft_{k},z_{ijk}} & \epsilon \cdot v_{N+1} + \frac{1}{N}\sum\limits_{i\in[N]}v_{i}, \vspace{1mm}\\
\;s.t. & \norm{\widetilde{\mfy}_{ij} - \mfx_{k}}^{2} - v_{N+1}\cdot \norm{\widetilde{\mfy}_{ij} - \mfy_{i}} - v_{i} \leq Mz_{ijk},  \vspace{1mm} \\
& \left[\begin{array}{llll}
& \mfA\mfx_{k} \geq \mfb,\; \mfu_{k} \geq \zero \vspace{1mm}\\
& \mfu_{k} \leq M\mft_{k} \vspace{1mm} \\
& \mfA\mfx_{k} - \mfb \leq M(1 - \mft_{k})  \vspace{1mm} \\
& (w_{k}^{1}Q_{1} + \cdots + w_{k}^{p}Q_{p})\mfx_{i} + w_{k}^{1}\mfc_{1} + \cdots + w_{k}^{p}\mfc_{p} - \mfA^{T}\mfu_{k} = 0 \vspace{1mm}
\end{array} \right] & \forall k\in[K] \vspace{1mm} \\
& 0 \leq \vartheta_{k} \leq M_{k}z_{k} & \forall k\in[K] \vspace{1mm} \\
& \mfx_{k} - M_{k}(1 - z_{k}) \leq \vartheta_{k} \leq \mfx_{k} & \forall k\in[K] \vspace{1mm} \\
& \sum\limits_{k\in[K]}z_{ijk} = K - 1, \vspace{1mm} \\
& \theta \in \Theta, v \in \mathcal{V}, \;\;\mfu_{k} \in \bR^{q}_{+}, \;\; \mft_{k} \in \{0,1\}^{q}, \;\;  z_{ijk} \in \{0,1\}, \forall i\in[N], j \in [J_{i}], k \in [K].
\end{array}
\end{align*}

\section{Omitted Proofs}

\subsection{Example \ref{example:lipschitz}}
We will show next that $ \kappa = 2R $.
Let $ \theta_{1} = (Q_{1}^{1},Q_{2}^{1},\mfc_{1}^{1},\mfc_{2}^{1}), \theta_{2} = (Q_{1}^{2},Q_{2}^{2},\mfc_{1}^{1},\mfc_{2}^{2}) $. Then,
\begin{align*}
\begin{array}{llll}
h(\mfx,w,\theta_{1},\theta_{2}) & = w^{T}\mff(\mfx, \theta_{1}) - w^{T}\mff(\mfx,\theta_{2}) \\
& = \frac{1}{2}w_{1}\mfx^{T} Q_{1}^{1}\mfx + w_{1}\mfx^{T}\mfc_{1}^{1} + \frac{1}{2}w_{2}\mfx^{T} Q_{2}^{1}\mfx + w_{2}\mfx^{T}\mfc_{2}^{1} \\
& \;\;\;\;- \frac{1}{2}w_{1}\mfx^{T} Q_{1}^{2}\mfx - w_{1}\mfx^{T}\mfc_{1}^{2} - \frac{1}{2}w_{2}\mfx^{T} Q_{2}^{2}\mfx - w_{2}\mfx^{T}\mfc_{2}^{2}.
\end{array}
\end{align*}
Since $ h(\mfx,w,\theta_{1},\theta_{2}) $ is continuously differentiable in $ \mfx $, the Lipschitz constant can be estimated by bounding the norm of the gradient of $ h $. We have
\begin{align*}
\frac{\partial h}{\partial \mfx} = w_{1}(Q_{1}^{1} - Q_{1}^{2})\mfx + w_{1}(\mfc_{1}^{1} - \mfc_{1}^{2}) + w_{2}(Q_{2}^{1} - Q_{2}^{2})\mfx + w_{2}(\mfc_{2}^{1} - \mfc_{2}^{2}).
\end{align*}
Thus,
\begin{align*}
\begin{array}{llll}
\sup\limits_{\norm{\mfx} \leq R}\norm{\frac{\partial h}{\partial \mfx}} & = \sup\limits_{\norm{\mfx} \leq R}\norm{w_{1}(Q_{1}^{1} - Q_{1}^{2})\mfx + w_{1}(\mfc_{1}^{1} - \mfc_{1}^{2}) + w_{2}(Q_{2}^{1} - Q_{2}^{2})\mfx + w_{2}(\mfc_{2}^{1} - \mfc_{2}^{2})} \\
& \leq \sup\limits_{\norm{\mfx} \leq R}\norm{w_{1}(Q_{1}^{1} - Q_{1}^{2})\mfx} + \norm{w_{1}(\mfc_{1}^{1} - \mfc_{1}^{2})} \\
& \;\;\;\;+ \sup\limits_{\norm{\mfx} \leq R}\norm{w_{2}(Q_{2}^{1} - Q_{2}^{2})\mfx} + \norm{w_{2}(\mfc_{2}^{1} - \mfc_{2}^{2})} \\
& \leq w_{1}\norm[F]{Q_{1}^{1} - Q_{1}^{2}}\cdot\sup\limits_{\norm{\mfx} \leq R}\norm{\mfx} + w_{1} \norm{\mfc_{1}^{1} - \mfc_{1}^{2}} \\
& \;\;\; + w_{2}\norm[F]{Q_{2}^{1} - Q_{2}^{2}}\cdot\sup\limits_{\norm{\mfx} \leq R}\norm{\mfx} + w_{2} \norm{\mfc_{2}^{1} - \mfc_{2}^{2}} \\
& \leq R\norm[F]{Q_{1}^{1} - Q_{1}^{2}} + \norm{\mfc_{1}^{1} - \mfc_{1}^{2}} + R\norm[F]{Q_{2}^{1} - Q_{2}^{2}} + \norm{\mfc_{2}^{1} - \mfc_{2}^{2}} \\
& \leq 2R\sqrt{\norm[F]{Q_{1}^{1} - Q_{1}^{2}}^{2} + \norm{\mfc_{1}^{1} - \mfc_{1}^{2}}^{2} + \norm[F]{Q_{2}^{1} - Q_{2}^{2}}^{2} + \norm{\mfc_{2}^{1} - \mfc_{2}^{2}}^{2}} \\
& = 2R\norm{\theta_{1} - \theta_{2}}.
\end{array}
\end{align*}
where the last inequality follows from the Power mean inequality. Hence, $ h(\cdot,w,\theta_{1},\theta_{2}) $ is $ 2R\norm{\theta_{1} - \theta_{2}}$-Lipschitz continuous on $ \mathcal{Y} $.

\subsection{Proof of Lemma \ref{lemma: uniformly continuous}}
\begin{proof}
	\begin{itemize}
		\item[(a)] Proof of (a) is straightforward. $ l_{K}(\mfy,\theta) = \min\limits_{\mfx \in \bigcup\limits_{k\in [K]}  S(w_{k},\theta)} \norm{\mfy - \mfx}^{2} \leq (\norm{\mfy} + B)^{2} \leq (R + B)^{2} $.
		
		\item[(b)] $ \forall \theta \in \Theta $, $ \forall \mfy_{1},\mfy_{2} \in \mathcal{Y} $, let
		\begin{align*}
		l_{K}(\mfy_{1},\theta) = \norm{\mfy_{1} - S(w^{1},\theta)}^{2},\; l_{K}(\mfy_{2},\theta) = \norm{\mfy_{2} - S(w^{2},\theta)}^{2}.
		\end{align*}
		
		Without of loss of generality, let $ l_{K}(\mfy_{1},\theta) \geq l_{K}(\mfy_{2},\theta) $. Then,
		\begin{align}\label{lips-in-y}
		\begin{array}{llll}
		|l_{K}(\mfy_{1},\theta) - l_{K}(\mfy_{2},\theta)| = l_{K}(\mfy_{1},\theta) - l_{K}(\mfy_{2},\theta) \vspace{1mm}\\
		= \norm{\mfy_{1} - S(w^{1},\theta)}^{2} - \norm{\mfy_{2} - S(w^{2},\theta)}^{2} \vspace{1mm}\\
		\leq \norm{\mfy_{1} - S(w^{2},\theta)}^{2} - \norm{\mfy_{2} - S(w^{2},\theta)}^{2} \vspace{1mm}\\
		= \langle \mfy_{1} - \mfy_{2}, \mfy_{1} + \mfy_{2} - 2S(w^{2},\theta) \rangle \vspace{1mm} \\
		\leq 2(B+R)\norm{\mfy_{1} - \mfy_{2}}
		\end{array}
		\end{align}
		
		The last inequality is due to Cauchy-Schwartz inequality and the Assumptions 3.1(a), that is
		\begin{align}\label{lips-in-y2}
		\norm{\mfy_{1} + \mfy_{2} - 2S(w^{2},\theta)} \leq 2(B+R)
		\end{align}
		
		Plugging \eqref{lips-in-y2}in \eqref{lips-in-y} yields the claim.
		
		\item[(c)] $ \forall \mfy \in \mathcal{Y} $, $ \forall \theta_{1},\theta_{2} \in \Theta $, let
		\begin{align*}
		l_{K}(\mfy,\theta_{1}) = \norm{\mfy - S(w^{1},\theta_{1})}^{2},\; l_{K}(\mfy,\theta_{2}) = \norm{\mfy_{2} - S(w^{2},\theta_{2})}^{2}.
		\end{align*}
		
		Without of loss of generality, let $ l_{K}(\mfy,\theta_{1}) \geq l_{K}(\mfy,\theta_{2}) $. Then,
		\begin{align}\label{lips}
		\begin{array}{llll}
		|l_{K}(\mfy,\theta_{1}) - l_{K}(\mfy,\theta_{2})| = l_{K}(\mfy,\theta_{1}) - l_{K}(\mfy,\theta_{2}) \vspace{1mm}\\
		= \norm{\mfy - S(w^{1},\theta_{1})}^{2} - \norm{\mfy - S(w^{2},\theta_{2})}^{2} \vspace{1mm}\\
		\leq \norm{\mfy - S(w^{2},\theta_{1})}^{2} - \norm{\mfy - S(w^{2},\theta_{2})}^{2} \vspace{1mm}\\
		= \langle S(w^{2},\theta_{2}) - S(w^{2},\theta_{1}), 2\mfy - S(w^{2},\theta_{1}) - S(w^{2},\theta_{2} ) \rangle \vspace{1mm} \\
		\leq 2(B+R)\norm{S(w^{2},\theta_{2}) - S(w^{2},\theta_{1})}
		\end{array}
		\end{align}
		
		The last inequality is due to Cauchy-Schwartz inequality and the Assumptions 3.1(a), that is
		\begin{align}\label{lips2}
		\norm{2\mfy - S(w^{2},\theta_{1}) - S(w^{2},\theta_{2})} \leq 2(B+R)
		\end{align}
		
		Next, we will apply Proposition 6.1 in \cite{bonnans1998optimization} to bound $ \norm{S(w^{2},\theta_{2}) - S(w^{2},\theta_{1})} $.
		
		Under Assumptions 3.1 - 3.2, the conditions of Proposition 6.1 in \cite{bonnans1998optimization} are satisfied. Therefore,
		\begin{align}\label{lipschitz2}
		\norm{S(w^{2},\theta_{2}) - S(w^{2},\theta_{1})}  \leq \frac{2\kappa}{\lambda}\norm{\theta_{1} - \theta_{2}}
		\end{align}
		
		Plugging \eqref{lips2} and \eqref{lipschitz2} in \eqref{lips} yields the claim.
	\end{itemize}
\end{proof}

\section{Data for the Portfolio Optimization Problem}
\renewcommand{\arraystretch}{.6}	
\begin{table}[ht]
	\centering
	\caption{True Expected Return}
	\label{table:true_return}
	\begin{tabular}{@{}ccccccccc@{}}
		\toprule
		Security        & 1      & 2      & 3      & 4      & 5      & 6      & 7      & 8      \\ \midrule
		Expected Return & 0.1791 & 0.1143 & 0.1357 & 0.0837 & 0.1653 & 0.1808 & 0.0352 & 0.0368 \\ \bottomrule
	\end{tabular}
\end{table}
\begin{table}[ht]
	\centering
	\caption{True Return Covariances Matrix}
	\label{table:true_covariance}
	\begin{tabular}{@{}ccccccccccc@{}}
		\toprule
		Security & 1      & 2      & 3      & 4      & 5      & 6      & 7      & 8      \\ \midrule
		1        & 0.1641 & 0.0299 & 0.0478 & 0.0491 & 0.058  & 0.0871 & 0.0603 & 0.0492 \\
		2        & 0.0299 & 0.0720  & 0.0511 & 0.0287 & 0.0527 & 0.0297 & 0.0291 & 0.0326 \\
		3        & 0.0478 & 0.0511 & 0.0794 & 0.0498 & 0.0664 & 0.0479 & 0.0395 & 0.0523 \\
		4        & 0.0491 & 0.0287 & 0.0498 & 0.1148 & 0.0336 & 0.0503 & 0.0326 & 0.0447 \\
		5        & 0.0580  & 0.0527 & 0.0664 & 0.0336 & 0.1073 & 0.0483 & 0.0402 & 0.0533 \\
		6        & 0.0871 & 0.0297 & 0.0479 & 0.0503 & 0.0483 & 0.1134 & 0.0591 & 0.0387 \\
		7        & 0.0603 & 0.0291 & 0.0395 & 0.0326 & 0.0402 & 0.0591 & 0.0704 & 0.0244 \\
		8        & 0.0492 & 0.0326 & 0.0523 & 0.0447 & 0.0533 & 0.0387 & 0.0244 & 0.1028 \\ \bottomrule
	\end{tabular}
\end{table}


\bibliographystyle{unsrt}
\bibliography{reference}   


\end{document}